\documentclass[
reqno,
10pt,
oneside
]{article}

\usepackage[ngerman, english]{babel}
\usepackage[utf8]{inputenc}
\usepackage{enumerate}
\usepackage[normalem]{ulem}
\usepackage[babel,french=guillemets,german=swiss]{csquotes}
\usepackage[center,font=small]{caption}
\usepackage{cite}
\usepackage{bbm}
\usepackage[raggedright]{titlesec}
\usepackage{xcolor}

\titleformat{\section}{\normalfont\large\bfseries}{\thesection}{1em}{}
\titleformat{\subsection}{\normalfont\bfseries}{\thesubsection}{1em}{}

\usepackage{tocbasic}
\DeclareTOCStyleEntry[
beforeskip=.2em plus 1pt,
]{tocline}{section}

\usepackage[%
left=1.75in,
right=1.75in,
bottom=1.25in,
top=1in,
footskip=0.5in,
]{geometry}

\usepackage{color}
\definecolor{LinkColor}{rgb}{0,0,1}
\definecolor{LinkColor2}{rgb}{0,0.5,0}
\definecolor{lbcolor}{rgb}{0.85,0.85,0.85}
\definecolor{FrameColor}{rgb}{0.85,0.85,0.85}
\definecolor{rosso}{rgb}{0.8,0,0}
\definecolor{lightgray}{rgb}{0.5,0.5,0.5}
\definecolor{violet}{rgb}{0.65,0,0.65}
\definecolor{darkgreen}{rgb}{0,0.5,0}

\usepackage{enumitem}

\usepackage{graphicx}
\usepackage{placeins}
\usepackage{overpic}

\usepackage{amsmath}
\usepackage{amssymb}
\usepackage{dsfont}
\usepackage{empheq}
\usepackage{amsthm}
\numberwithin{equation}{section}

\usepackage[%
pdftitle={Titel},%
pdfauthor={Autor},%
pdfcreator={LaTeX, LaTeX with hyperref and KOMA-Script},
pdfsubject={Betreff}, 
pdfkeywords={Keywords}
]{hyperref} 

\hypersetup{%
	colorlinks	=true,
	linkcolor	=LinkColor,%
	anchorcolor	=LinkColor,%
	citecolor	=LinkColor2,%
	filecolor	=LinkColor,%
	menucolor	=LinkColor,%
	urlcolor	=LinkColor,%
}

\newtheorem{theorem}{Theorem}[section]
\newtheorem{lemma}[theorem]{Lemma}
\newtheorem{proposition}[theorem]{Proposition}
\newtheorem{corollary}[theorem]{Corollary}
\newtheorem{definition}[theorem]{Definition}

\theoremstyle{definition}
\newtheorem{remark}[theorem]{Remark}

\makeatletter
\renewenvironment{proof}[1][\proofname]{%
	\par\pushQED{\qed}\normalfont%
	\topsep6\p@\@plus6\p@\relax
	\trivlist\item[\hskip\labelsep\bfseries#1\@addpunct{.}]%
	\ignorespaces
}{%
	\popQED\endtrivlist\@endpefalse
}
\makeatother

\makeatletter
\renewcommand\paragraph{\@startsection{paragraph}{4}{\z@}%
	{1ex \@plus1ex \@minus.2ex}%
	{-1em}%
	{\normalfont\normalsize\bfseries}}
\renewcommand\subparagraph{\@startsection{paragraph}{4}{\z@}%
	{1ex \@plus1ex \@minus.2ex}%
	{-1em}%
	{\normalfont\normalsize\itshape}}
\makeatother

\newcommand{\abs}[1]{\left| #1 \right|}
\newcommand{\bigabs}[1]{\big| #1 \big|}

\newcommand{\norm}[1]{\| #1 \|}
\newcommand{\bignorm}[1]{\big\| #1 \big\|}
\newcommand{\ang}[2]{ \langle #1 , #2  \rangle}
\newcommand{\bigang}[2]{ \big< #1 , #2  \big>}
\newcommand{\scp}[2]{ \left( #1 , #2  \right)}

\newcommand{\mean}[1]{\langle #1 \rangle}

\newcommand{\R}{\mathbb R}
\newcommand{\N}{\mathbb N}
\newcommand{\n}{\mathbf{n}}

\newcommand{\intO}{\int_\Omega}

\newcommand{\ep}{\varepsilon}
\newcommand{\supp}{\textnormal{supp\,}}
\newcommand{\dist}{\textnormal{dist\,}}

\newcommand{\dtau}{\;\mathrm d\tau}
\newcommand{\dx}{\;\mathrm d\mathbf{x}}
\newcommand{\dy}{\;\mathrm d\mathbf{y}}
\newcommand{\dt}{\;\mathrm dt}
\newcommand{\ds}{\;\mathrm ds}
\newcommand{\dtx}{\;\mathrm d(t,\mathbf{x})}
\newcommand{\dsx}{\;\mathrm d(s,\mathbf{x})}

\newcommand{\dds}{\frac{\mathrm d}{\mathrm ds}}

\newcommand{\del}{\partial}
\newcommand{\delt}{\partial_{t}}
\newcommand{\delth}{\partial_{t}^{h}}

\newcommand{\Grad}{\nabla}

\newcommand{\PN}{(-\Delta_N)^{-1}}

\newcommand{\emb}{\hookrightarrow}
\newcommand{\ssubset}{\subset\joinrel\subset}
\newcommand{\ov}{\overline}
\newcommand{\suchthat}{\;\ifnum\currentgrouptype=16 \middle\fi|\;}

\newcommand{\p}{\mathbf{p}}
\newcommand{\q}{\mathbf{q}}
\newcommand{\x}{\mathbf{x}}
\newcommand{\y}{\mathbf{y}}
\newcommand{\e}{\mathbf{e}}
\newcommand{\zero}{\mathbf{0}}

\newcommand{\Om}{\Omega}

\newcommand{\OT}{\Omega_T}
\newcommand{\Ot}{\Omega_{t}}

\newcommand{\revised}[1]{#1}

\newcommand{\Bigabs}[1]{\Big| #1 \Big|}
\newcommand{\biggabs}[1]{\bigg| #1 \bigg|}

\newcommand{\dkh}{\nabla_k^h}
\newcommand{\dkmh}{\nabla_k^{-h}}
\newcommand{\dnh}{\nabla_n^h}

\newcommand{\oton}{\{1,\dots,n\}}
\newcommand{\otonmo}{\{1,\dots,n-1\}}

\begin{document}

%
%

\title{\bfseries The anisotropic Cahn--Hilliard equation:\\ regularity theory and\\ strict separation properties
\\[-1.5ex]$\;$}

\author{Harald Garcke \footnotemark[1] 
		\and Patrik Knopf \footnotemark[1]
        \and Julia Wittmann \footnotemark[1]}

\date{ }

\maketitle

\begin{center}
    \bfseries
    \textit{This paper is dedicated to the 65th birthday of Professor Pierluigi Colli}
\end{center}

\renewcommand{\thefootnote}{\fnsymbol{footnote}}

\footnotetext[1]{
    Faculty for Mathematics, 
    University of Regensburg, 
    93053 Regensburg, 
    Germany \newline
	\tt(%
        \href{mailto:harald.garcke@ur.de}{harald.garcke@ur.de},
        \href{mailto:patrik.knopf@ur.de}{patrik.knopf@ur.de},
        \href{mailto:julia4.wittmann@ur.de}{julia4.wittmann@ur.de}%
        ).
}

\begin{center}
	\scriptsize
	{
		\textit{This is a preprint version of the paper. Please cite as:} \\  
		H.~Garcke, P.~Knopf, J.~Wittmann, 
        \textit{Discrete Contin.~Dyn.~Syst.~Ser.~S}
        \textbf{16}(12): 3622-3660 (2023)\\
		\url{https://doi.org/10.3934/dcdss.2023146}
    }
\end{center}

\smallskip

%
%

\begin{small}
\begin{center}
    \textbf{Abstract}
\end{center}
The Cahn--Hilliard equation with anisotropic energy contributions frequently appears in many physical systems. Systematic analytical results for the case with the relevant logarithmic free energy have been missing so far. We close this gap and show existence, uniqueness, regularity, and separation properties of weak solutions to the anisotropic Cahn--Hilliard equation with logarithmic free energy. Since firstly, the equation becomes highly non-linear, and secondly, the relevant anisotropies are non-smooth, the analysis becomes quite involved. In particular, new regularity results for quasilinear elliptic equations of second order need to be shown.
\\[1ex]
\\[1ex]	
\end{small}


\setlength\parindent{0ex}
\setlength\parskip{1ex}
\allowdisplaybreaks


\section{Introduction} 
\label{SECT:INTRO}
The classical Cahn--Hilliard equation
\begin{subequations}
\label{CH} 
\begin{alignat}{2}
    \label{CH:1}
    \delt \varphi &= \Grad \cdot \big(M(\varphi) \Grad\mu \big)
    &&\quad\text{in $\OT$},\\
    \label{CH:2}
    \mu &= -\ep \Delta \varphi + \ep^{-1} F'(\varphi)
    &&\quad\text{in $\OT$},
\end{alignat}
\end{subequations} 
where $\Omega\subset \R^d$ \revised{with $d\in \N$} 
is a bounded domain, $T$ is a positive final time and
$\Omega_T=(0,T)\times \Omega$, can be considered as an $H^{-1}$-gradient flow
of the
Ginzburg--Landau energy
\begin{align*}
  E_\mathrm{iso}(\varphi) =
  \int_\Omega \frac{\varepsilon}{2}|\nabla\varphi|^2 +
    \ep^{-1} F(\varphi) \dx,
\end{align*}
see \cite{Fife93, GarckeDMV}.
Here, $\varepsilon>0$ is a small parameter related to the width of a diffuse
interface separating two different phases, and $F:\R\to\R$ is a free energy density. The function $F$
has two global minima that correspond to these two phases, and the 
functional $E_\mathrm{iso}$ describes the energy associated with the diffuse interface. This interfacial energy, which is given by the
Ginzburg--Landau functional, is isotropic. This means that the
energy does not depend on the local orientation of the interface. In many
applications, in particular when crystalline materials are to be described, anisotropic
energies are more realistic, see, e.g., \cite{Taylor94linking}. In anisotropic materials, the energy
density actually depends on the local orientation of the interface. In fact, in
a sharp-interface description, the energy is given by
\begin{align}\label{SIen}
  \int_\Sigma \gamma(\boldsymbol{\nu})\; \mathrm d\mathcal{H}^{d-1},
\end{align}
where $\Sigma$ is a hypersurface representing the interface,
$\boldsymbol{\nu}$ stands for the unit normal on $\Sigma$,
$\mathrm d\mathcal{H}^{d-1}$ refers to  integration with respect to the
$(d-1)$-dimensional surface measure, and
\begin{align*}
  \gamma: \R^d\to (0,\infty)
\end{align*}
is an anisotropic density function. The energy density in \eqref{SIen} now depends via
the normal $\boldsymbol{\nu}$ on the local orientation of the
interface. In this context, the function $\gamma$ is assumed to be convex and to fulfill
\begin{subequations}
\begin{alignat}{2}
    \label{gamma:1}
    \gamma(\lambda \p)&= \lambda\gamma(\p)
    &&\quad\text{for all $\p\in\R^d\setminus\{\mathbf{0}\}$ and all $\lambda>0$},\\
    \label{gamma:2}
    \gamma(\p)&>0 
    &&\quad\text{for all $\p\in\R^d\setminus\{\mathbf{0}\}$}.
\end{alignat}
\end{subequations} 
The property \eqref{gamma:1} says that $\gamma$ is positively
one-homogeneous and implies that an anisotropic $\gamma$ is
necessarily not differentiable in $\p=\mathbf{0}$.
On the diffuse-interface level, the energy is then given by
\begin{align}\label{DiffEn}
  E:H^1(\Omega)\to\R, \quad E(\varphi)= \int_\Omega \varepsilon
  A(\nabla\varphi) + \varepsilon^{-1} F(\varphi) \dx,
\end{align}
where $A$ is given by
\begin{align*}
    A(\p) = \tfrac{1}{2} \gamma^2(\p)
    \quad\text{for all $\p\in\R^d$}.
\end{align*}
It can be shown that the diffuse-interface energy $E$ introduced in \eqref{DiffEn} converges to the sharp-interface energy \eqref{SIen} up to a multiplicative constant in the sense of $\Gamma$-limits, see, e.g., \cite{Garcke2023,Barroso1994,Bellettini2005,Bouchitte1990,Owen1991}.
For the forthcoming mathematical analysis, it is important that $A$ is
positively two-homogeneous. In particular, this already entails that an anisotropic $A$ necessarily
cannot be twice differentiable. This will make the regularity theory,
which will be discussed later, quite involved. 

As discussed for example in \cite{GKNZ}, one can also consider the
$H^{-1}$-gradient flow of the anisotropic energy $E$ introduced in \eqref{DiffEn}.
In fact, a weighted $H^{-1}$-gradient flow approach was used in \cite{GKNZ} to study the following general version of the anisotropic Cahn--Hilliard equation:
\begin{subequations}
\label{AICH} 
\begin{alignat}{2}
    \label{AICH:1}
    \delt \varphi &= \Grad \cdot \big(M(\Grad\varphi,\varphi) \Grad\mu \big)
    &&\quad\text{in $\OT$},\\
    \label{AICH:2}
    \mu &= -\ep \Grad\cdot A'(\Grad\varphi) + \ep^{-1} F'(\varphi)
    &&\quad\text{in $\OT$},\\
    \label{AICH:3}
    \ep A'(\Grad\varphi)\cdot \n &= 0
    &&\quad\text{on $\Gamma_T$},\\
    \label{AICH:4}
    \Grad\mu\cdot \n &= 0
    &&\quad\text{on $\Gamma_T$},\\
    \label{AICH:5}
    \varphi\vert_{t=0} &= \varphi_0
    &&\quad\text{in $\Omega$},
\end{alignat}
\end{subequations} 
where $\Gamma_T=(0,T)\times \partial \Omega$, and $\n$ is the outer unit normal on $\partial \Omega$. 
From now on, we assume that $\Omega\subset \R^d$ with $d\in \{2,3\}$ is a bounded domain with Lipschitz boundary. 
Moreover, the function $M:\R^d\times\R\to\R^+$ representing the mobility of the mixture is also allowed to be anisotropic.  

Especially in applications related to materials science, the physically relevant choice of $F$ in the anisotropic energy $E$ is the \textit{logarithmic potential}
\begin{align}
    \label{DEF:F:LOG}
    F(s) = \frac{\theta}{2} \big[(1+s)\,\ln(1+s) +(1-s)\,\ln(1-s)\big]+\frac{\theta_c}{2}(1-s^2),
\end{align}
for all $s\in (-1,1)$, which is also referred to as the \textit{Flory--Huggins potential}. Here, ${\theta>0}$ is the absolute temperature of the mixture, and $\theta_c$ is a critical temperature such that phase separation occurs if $0<\theta<\theta_c$. The logarithmic potential is classified as a singular potential as its derivative $F'$ diverges to $\pm\infty$ when its argument approaches $\pm 1$. It is often approximated by the \textit{polynomial double-well potential}
\begin{align}
    \label{DEF:F:REG}
    F(s) = \frac{\alpha}{4}(s^2-1)^2
    \quad\text{for all $s\in (-1,1)$},
\end{align}
where $\alpha>0$ is a suitable constant. In the following, we will simply set $\alpha=1$ for simplicity.
Another common singular potential is the \textit{double-obstacle potential}, which is given by
\begin{align}
    \label{DEF:F:DOB}
    F(s)=
    \begin{cases}
    \frac{1}{2}(1-s^2)&\text{if $|s|\leq 1$,}\\
    +\infty&\text{else}.
    \end{cases}
\end{align}

Although originally introduced to model spinodal decomposition in binary alloys, the Cahn--Hilliard equation found many new applications in the recent past. 
The anisotropic Cahn--Hilliard equation has also been used to describe anisotropic solidification processes. In particular, it is capable of describing snow crystal growth as well as the solidification of metals, see, e.g., \cite{BGN2013stable, vch}. Another important application of the anisotropic Cahn--Hilliard equation is to model the growth of thin solid films, which play an important role in the self-organization of nanostructures, see, e.g.,
\cite{Ratz2006surface, Dziwnik17anisotropic, Torabi209, GKNZ}. In this context, also works on the anisotropic Allen--Cahn equation \cite{
Elliott96limit, Alfaro2010motion,
Graser2013time, laux2022diffuseinterface} and higher order models \cite{MR3918378} are relevant. For more information
on regularity theory and separation properties of the isotropic Cahn--Hilliard equation with logarithmic potential, we recommend the paper \cite{Gal2023} as well as the book \cite{Miranvillebook}.


The goal of the present paper is to show the existence of weak solutions to the anisotropic Cahn--Hilliard equation with logarithmic potential, to prove a uniqueness result, to show regularity results, and to establish a separation property 
\revised{under suitable assumptions on the initial data.
It states that for almost all $t\in[0,T]$, there exists $\delta(t) \in (0,1]$ such that
\begin{equation*}
    \norm{\varphi(t)}_{L^\infty(\Omega)} \le 1 - \delta(t).
\end{equation*}
Here, the term ``separation property'' means that the solution stays away from the pure states, which are represented by $\varphi =\pm 1$.
In the case $d=2$, we are even able to establish a separation property that is uniform in time. More precisely, there exists a constant $\delta^*\in (0,1]$ such that for all $(t,\x)\in \Omega_T$,}
\begin{equation*}
    \revised{\abs{\varphi(t,\x)} \le 1 - \delta^*.}
\end{equation*}
These results are already well-known in the case of an isotropic energy (see, e.g., \cite{Gal2023,Miranvillebook}), but become non-trivial in the case of anisotropic energies. This is firstly because the resulting equations are much more non-linear and secondly because the anisotropic potential $A$ is not smooth. The latter already follows from the requirement that $A$ is positively two-homogeneous, but can also be traced back to the usage of non-smooth anisotropies $\gamma$. It is therefore surprising that we can still prove $H^2$-regularity with respect to the spatial variables, and that we are also able to show that $\partial_t \varphi$ lies in $L^2(0,T;H^1(\Omega))$. To establish separation properties, typically a quite high regularity of the solutions is required. We are able to derive certain estimates for $F'(\varphi)$, which at least allows us to prove a separation property for almost all times.

The outline of the paper is as follows.
In Section~\ref{SECT:PRELIM}, we introduce the notation used in this paper, present the general assumptions and state the main results. In Section~\ref{SECT:EXLOG}, we prove the existence result for the anisotropic Cahn--Hilliard equation with a logarithmic potential. Section~\ref{SECT:UNIQ} is devoted to the uniqueness of weak solutions and in Section~\ref{SECT:REG:QLE}, we present regularity results for a general class of quasilinear second-order elliptic equations under minimal assumptions on the data. These results are then used in Section~\ref{SECT:REG:AICH} to prove higher regularity of weak solutions to the anisotropic Cahn--Hilliard equation, which can in turn be used to establish a separation property for almost all times.


\section{Preliminaries and main results} 
\label{SECT:PRELIM}

\subsection{Notation} \label{SUBSECT:NOT}
We first introduce some notation that is supposed to hold throughout this paper.

For any $1 \leq p \leq \infty$ and $k \geq 0$, we use the standard notation $L^p(\Omega)$ and $W^{k,p}(\Omega)$ for Lebesgue and Sobolev spaces on any open set $\Omega \subset \R^d$ (with $d\in\N$). The corresponding norms are denoted by $\norm{\,\cdot\,}_{L^p(\Omega)}$ and $\norm{\,\cdot\,}_{W^{k,p}(\Omega)}$.  
In the case $p = 2$, these spaces are Hilbert spaces and we write $H^k(\Omega) = W^{k,2}(\Omega)$. As usual, we identify $H^0(\Omega)$ with $L^2(\Omega)$. 

For any Banach space $X$, its dual space is denoted by $X'$, and the associated duality pairing between elements $y\in X'$ and $x\in X$ is denoted by $\ang{y}{x}_X$. In the case $X=H^1(\Omega)$, we will simply write $\ang{\cdot}{\cdot}$ instead of $\ang{\cdot}{\cdot}_{H^1(\Omega)}$.
If $X$ is a Hilbert space, we denote its inner product by $(\cdot, \cdot)_X$. Moreover, we write
\begin{align*}
\mean{f}_\Omega := 
\frac{1}{\abs{\Omega}} \ang{f}{1}_{H^1(\Omega)} \quad \text{ for } f \in H^1(\Omega)' 
\end{align*}
to denote the generalized spatial mean of $f$. Here, $\abs{\Omega}$ denotes the $d$-dimensional Lebesgue measure of $\Omega$.  
With the usual identification $L^1(\Omega) \subset H^1(\Omega)'$ it holds that
$\mean{f}_\Omega = \frac{1}{\abs{\Omega}} \int_\Omega f \dx$ if 
$f \in L^1(\Omega)$.
In addition, we introduce 
\begin{align*}
    H^1_{(m)}(\Omega) &:= \big\{ u \in H^1(\Omega) \,\big\vert\, \mean{u}_\Om = m \big\}
    \qquad\text{for any $m\in\R$},
    \\
    H^{-1}_{(0)}(\Omega) &:= \big\{ f \in \big(H^1(\Omega)\big)' \,\big\vert\, \mean{f}_\Om = 0 \big\}.
\end{align*}
We point out that for every $m\in\R$, $H^1_{(m)}(\Omega)$ is an affine subspace of the Hilbert space $H^1(\Omega)$. In the case $m=0$, it is even a closed linear subspace, meaning that $H^1_{(0)}(\Omega)$ is also a Hilbert space.

Let us now assume that $\Omega\subset \R^d$ with $d\in\{2,3\}$ is a bounded Lipschitz domain. It is well-known that there exists a solution operator \begin{align*}
    \PN:H^{-1}_{(0)}(\Omega) \to H^1_{(0)}(\Omega), \quad f \mapsto u_f
\end{align*}
mapping any functional $f\in H^{-1}_{(0)}(\Omega)$ onto its corresponding solution $u_f\in H^1_{(0)}(\Omega)$ of the Poisson--Neumann problem
\begin{align*}
    -\Delta u_f = f\quad\text{in $\Omega$}, \qquad \Grad u_f \cdot\n = 0 \quad\text{on $\Gamma = \partial\Omega$}.
\end{align*}
Then, the bilinear form
\begin{align*}
    &\scp{\,\cdot\,}{\,\cdot\,}_{-1}: H^{-1}_{(0)}(\Omega) \times H^{-1}_{(0)}(\Omega) \to \R,
    \\
    &\scp{f}{g}_{-1} := \scp{\Grad\PN f}{\Grad\PN g}_{L^2(\Omega)}
\end{align*}
defines an inner product on $H^{-1}_{(0)}(\Omega)$, and the function
\begin{align*}
    \norm{\,\cdot\,}_{-1}: H^{-1}_{(0)}(\Omega) \to \R,
    \quad
    \norm{f}_{-1} := \sqrt{\scp{f}{f}_{-1}}
\end{align*}
defines the induced norm on $H^{-1}_{(0)}(\Omega)$. Note that $H^{1}_{(0)}(\Omega) \subset H^{-1}_{(0)}(\Omega)$, and the functions $\scp{\,\cdot\,}{\,\cdot\,}_{-1}$ and $\norm{\,\cdot\,}_{-1}$ also define an inner product and a norm, respectively, on the space $H^{1}_{(0)}(\Omega)$.

\bigskip

\subsection{General assumptions} We make the following general assumptions.
\begin{enumerate}[label=$(\mathbf{A \arabic*})$, ref = $\mathbf{A \arabic*}$]
	\item \label{ass:dom} The set $\Omega \subset \R^d$ with $d \in \{2,3\}$ is a bounded Lipschitz domain with outer unit normal vector field $\n$.
	Moreover, the final time $T$ and the interface parameter $\ep$ are \revised{arbitrary} positive constants.  
    For simplicity, we set $\ep=1$ as the choice of this parameter does not have any impact on the mathematical analysis. 
	\item \label{ass:A} The function $A:\R^d\to \R$ is continuously differentiable, positive on $\R^d\setminus\{\mathbf 0\}$, and positively two-homogeneous, i.e.,
    \begin{align*}
        A(\lambda\p) = \lambda^2 A(\p)
        \quad \text{for all $\lambda>0$ and $\p\in\R^d$}.
    \end{align*}
    As a consequence, there exist constants $A_0,A_1,a_1\in\R$ with $0<A_0\le A_1$ and $a_1>0$ such that 
	\begin{align*}
	    A_0  \abs{\p}^2 \le A(\p) \le A_1  \abs{\p}^2
        \quad\text{and}\quad
        \abs{A'(\p)} \le a_1 \abs{\p}
	    \quad
        \text{for all $\p\in\R^d$}.
	\end{align*}
	The gradient $A':\R^d\to\R^d$ is assumed to be strongly monotone, i.e., there exists a constant $a_0>0$ such that
	\begin{align*}
	    \big(A'(\p)-A'(\q)\big)\cdot(\p-\q) \ge a_0 \abs{\p-\q}^2 
	    \quad\text{for all $\p,\q\in\R^d$}.
	\end{align*}
	This directly implies that $A$ is strongly convex and thus strictly convex. 
    \item \label{ass:M} The function $M:\R^d\times\R\to\R$ is continuous and there exist constants $M_0,M_1\in\R$ with $0<M_0\le M_1$ such that
	\begin{align*}
	    M_0 \le M(\p,s) \le M_1
	    \quad\text{for all $\p\in\R^d$ and $s\in\R$}.
	\end{align*}
\end{enumerate}

\subsection{Weak solutions for smooth potentials}

The anisotropic Cahn--Hilliard equation \eqref{AICH} endowed with a regular potential $F$ was investigated in \cite{GKNZ}. In this subsection, we recall the most important results.

We consider potentials $F$ satisfying the following assumption:
\begin{enumerate}[label=$(\mathbf{F \arabic*})$, ref = $\mathbf{F \arabic*}$]
	\item \label{ass:F1} 
	The potential $F:\R\to\R$ is continuously differentiable. 
	Moreover, there exists an exponent $p\in[2,6)$ as well as 
	non-negative constants $B_F$, $C_{F}$ and $C_{F'}$ such that
		\begin{align*}
		- B_{F} \le F(s) \le C_{F}(1+\abs{s}^p) 
		\quad\text{and}\quad
		\abs{F'(s)} \le C_{F'}(1+\abs{s}^{p-1})
		\end{align*}
	for all $s\in\R$.
\end{enumerate}
The polynomial double-well potential introduced in \eqref{DEF:F:REG} naturally fulfills \eqref{ass:F1} with $p=4$. 
However, both the logarithmic potential (see \eqref{DEF:F:LOG}) and the double-obstacle potential (see \eqref{DEF:F:DOB}) do not satisfy this assumption.

\pagebreak[2]

A weak solution of the anisotropic Cahn--Hilliard equation \eqref{AICH} is defined as follows.
\begin{definition} \label{DEF:WS}
    Suppose that the assumptions \eqref{ass:dom}--\eqref{ass:M} and \eqref{ass:F1} are fulfilled, and let ${\varphi_0\in H^1(\Omega)}$ be any initial datum.
    Then, the pair $(\varphi,\mu)$ is called a weak solution of the anisotropic Cahn--Hilliard equation \eqref{AICH} if the following conditions are fulfilled:
    \begin{enumerate}[label=\textnormal{(\roman*)}, ref = \textnormal{(\roman*)}]
        \item \label{DEF:WS:REG} 
        The functions $\varphi$, $\mu$ and $F'(\varphi)$ have the regularity
        \begin{alignat*}{2}
            \varphi &\in C\big([0,T];L^2(\Omega)\big) 
                \cap L^\infty\big(0,T;H^1(\Omega)\big)
                \cap H^1\big(0,T;H^1(\Omega)'\big),
            \\
            \mu &\in L^2\big(0,T;H^1(\Omega)\big),
            \\
            F'(\varphi) &\in L^2\big(0,T;L^2(\Omega)\big).
        \end{alignat*}
        \item \label{DEF:WS:WEAK} 
        The pair $(\varphi,\mu)$ satisfies the weak formulations
        \begin{subequations}
        \label{DEF:WS:WF}
        \begin{alignat}{2}
            \label{DEF:WS:WF1}
            \bigang{\delt\varphi}{\zeta}_{H^1(\Omega)} 
            &= - \intO M(\Grad\varphi,\varphi)\, \Grad \mu \cdot \Grad \zeta \dx,
            \\
            \label{DEF:WS:WF2}
            \intO \mu\, \eta \dx 
            &= \intO A'(\Grad\varphi)\cdot \Grad \eta + F'(\varphi)\,\eta \dx
        \end{alignat}
        \end{subequations}
        a.e.~in $[0,T]$ for all test functions $\zeta,\eta\in H^1(\Omega)$. Moreover, $\varphi$ satisfies the initial condition 
        \begin{align}
            \label{DEF:WS:INI}
            \varphi(0) = \varphi_0 \qquad\text{a.e.~in $\Omega$.}
        \end{align}
        \item \label{DEF:WS:ENERGY}
        The pair $(\varphi,\mu)$ satisfies the weak energy dissipation law
        \begin{align}
            \label{DEF:WS:DISS}
            E\big(\varphi(t)\big) 
            + \frac 12 \int_0^t \intO M(\Grad\varphi,\varphi)\, \abs{\Grad\mu}^2 \dx\ds
            \le E(\varphi_0)
        \end{align}
        for almost all $t\in[0,T]$.
    \end{enumerate}
\end{definition}

\medskip
\pagebreak[2]

The existence of such a weak solution was established in \cite{GKNZ}. We restate the existence result as the following proposition.
\begin{proposition} \label{PROP:REGPOT}
    Suppose that the assumptions \eqref{ass:dom}--\eqref{ass:M} and \eqref{ass:F1} are fulfilled, and let ${\varphi_0\in H^1(\Omega)}$ be any initial datum. Then, there exists a weak solution $(\varphi,\mu)$ to the system \eqref{AICH} in the sense of Definition~\ref{DEF:WS}.
\end{proposition}

\revised{We point out that since the final time $T$ can be chosen arbitrarily (see \eqref{ass:dom}), the weak solution provided by Proposition~\ref{PROP:REGPOT} actually exists globally in time, i.e., on the whole interval $[0,\infty)$.}
To obtain an additional uniform estimate for $F'(\varphi)$, we need a stronger assumption on the potential $F$:

\begin{enumerate}[label=$(\mathbf{F \arabic*})$, ref = $\mathbf{F \arabic*}$, start=2]
	\item \label{ass:F2} 
    The potential $F:\R\to\R$ is twice continuously differentiable, non-negative,
	and there exist constants $c_0,c_1 \ge 0$ such that 
    \begin{align}
    \label{COND:F''}
        - c_0 \le F''(s) \le c_1 \quad\text{for all $s\in\R$}.
    \end{align}
\end{enumerate}

We point out that \eqref{ass:F2} implies that \eqref{ass:F1} holds with $p=2$, which means that the potential $F$ has at most quadratic growth.
Provided that $F$ satisfies \eqref{ass:F2}, we obtain the following uniform estimate on $F'(\varphi)$.

\begin{corollary} \label{COR:REGPOT}
    Suppose that the assumptions \eqref{ass:dom}--\eqref{ass:M}
    and \eqref{ass:F2} are fulfilled. 
    Let $\varphi_0\in H^1(\Omega)$ be any initial datum satisfying 
    $\abs{\mean{\varphi_0}_\Omega} \le 1 - \kappa$ for some $\kappa\in(0,1]$,
    and let $(\varphi,\mu)$ be a corresponding weak solution. Then, there exists a constant $c>0$ depending only on $\varphi_0$, $E(\varphi_0)$, $c_0$ and the constants in \eqref{ass:dom}--\eqref{ass:M}, but not on $c_1$, such that for almost all $t\in [0,T]$,
    \begin{align}
    \label{EST:F'}
        \bignorm{F'\big(\varphi(t)\big)}_{L^2(\Omega)}^2 \le \frac{c}{\kappa^2}
        \Big[ 1 & + \norm{F}_{L^\infty([-R,R])}^2 + \norm{\Grad\mu(t)}_{L^2(\Omega)}^2\Big],
    \end{align}
    where $ R:= \abs{\mean{\varphi_0}_\Omega} + \frac{\kappa}{2} < 1$.
\end{corollary}

A proof of this result can be found in \cite{GKNZ}. We point out that the inequality \eqref{EST:F'} was not explicitly stated there but can clearly be extracted from the proof of \cite[Corollary~4.4]{GKNZ}. 

The connection between \eqref{EST:F'} and the uniform estimate stated in \cite[Corollary~4.4]{GKNZ} is as follows: 
Due to \eqref{ass:M}, \eqref{ass:F2} and the weak energy dissipation law \eqref{DEF:WS:DISS}, we obtain the bound
\begin{align}
    \label{EST:UNIFORM}
    \frac 12 M_0 \norm{\Grad\mu}_{L^2(0,T;L^2(\Omega))}^2 \le E(\varphi_0).
\end{align}
We now integrate \eqref{EST:F'} with respect to time from $0$ to $T$. Employing \eqref{EST:UNIFORM}, we conclude
\begin{align}
\label{EST:F':INT}
    \bignorm{F'\big(\varphi\big)}_{L^2(\OT)}^2 \le \frac{c^*}{\kappa^2}
    \Big[ 1 + \norm{F}_{L^\infty([-R,R])}^2 \Big],
\end{align}
for some constant $c^*>0$ depending only on $\varphi_0$, $E(\varphi_0)$, $c_0$ and the constants in \eqref{ass:dom}--\eqref{ass:M}, which is exactly the estimate stated in \cite[Corollary~4.4]{GKNZ}.

We point out that the more precise estimate \eqref{EST:F'} will be needed in the context of Theorem~\ref{THM:REG:2}, where $F'(\varphi)\in L^\infty\big(0,T;L^2(\Omega)\big)$ will be shown under suitable assumptions.

\subsection{Weak solutions for the logarithmic potential}

In this subsection, we choose $F$ as the logarithmic potential that was introduced in \eqref{DEF:F:LOG}. It can be decomposed as $F(s) = F_1(s) + F_2(s)$,
where
\begin{subequations}
\label{DEC:LOG}
\begin{align}
    F_1(s) &= \frac{\theta}{2} \big[(1+s)\,\ln(1+s) +(1-s)\,\ln(1-s)\big],
    \\
    F_2(s) &= \frac{\theta_c}{2}(1-s^2)
\end{align}
\end{subequations}
for all $s\in (-1,1)$. We point out that $F_1$ is convex whereas $F_2$ is concave. Consequently, the derivative $F_1'$ is monotonically increasing. It is worth mentioning that the function $F_1$ (and thus also $F$) possesses a continuous extension onto the interval $[-1,1]$. However, this does not hold for $F_1'$ as this function becomes singular at $\pm 1$.

In this situation, a weak solution of the anisotropic Cahn--Hilliard equation \eqref{AICH} is defined as follows.
\begin{definition} \label{DEF:WS*}
    Suppose that the assumptions \eqref{ass:dom}--\eqref{ass:M} are fulfilled, let $F$ be the logarithmic potential given by \eqref{DEF:F:LOG}, and let ${\varphi_0\in H^1(\Omega)}$ be any initial datum with $\abs{\varphi_0}\le 1$ a.e.~in $\Omega$ as well as $\abs{\mean{\varphi_0}}<1$.
    Then, the pair $(\varphi,\mu)$ is called a weak solution of the anisotropic Cahn--Hilliard equation \eqref{AICH} if the following conditions are fulfilled: \vspace{-1ex}
    \begin{enumerate}[label=\textnormal{(\roman*)}, ref = \textnormal{(\roman*)}]
        \item \label{DEF:WS:REG*} 
        The functions $\varphi$, $\mu$ and $F'(\varphi)$  have the regularity
        \begin{alignat*}{2}
            \varphi &\in C\big([0,T];L^2(\Omega)\big) 
                \cap L^\infty\big(0,T;H^1(\Omega)\big)
                \cap H^1\big(0,T;H^1(\Omega)'\big),
            \\
            \mu &\in L^2\big(0,T;H^1(\Omega)\big),
            \\
            F'(\varphi) &\in L^2\big(0,T;L^2(\Omega)\big),
        \end{alignat*}
        and it holds $\abs{\varphi}<1$ almost everywhere in $\OT$.
        \item \label{DEF:WS:WEAK*} 
        The pair $(\varphi,\mu)$ satisfies the weak formulations
        \begin{subequations}
        \label{DEF:WS:WF*}
        \begin{alignat}{2}
            \label{DEF:WS:WF1*}
            \bigang{\delt\varphi}{\zeta}_{H^1(\Omega)} 
            &= - \intO M(\Grad\varphi,\varphi)\, \Grad \mu \cdot \Grad \zeta \dx,
            \\
            \label{DEF:WS:WF2*}
            \intO \mu\, \eta \dx 
            &= \intO A'(\Grad\varphi)\cdot \Grad \eta + F'(\varphi)\,\eta \dx
        \end{alignat}
        \end{subequations}
        a.e.~in $[0,T]$ for all test functions $\zeta,\eta\in H^1(\Omega)$. Moreover, $\varphi$ satisfies the initial condition 
        \begin{align}
            \label{DEF:WS:INI*}
            \varphi(0) = \varphi_0 \qquad\text{a.e.~in $\Omega$.}
        \end{align}
        \item \label{DEF:WS:ENERGY*}
        The pair $(\varphi,\mu)$ satisfies the weak energy dissipation law
        \begin{align}
            \label{DEF:WS:DISS*}
            E\big(\varphi(t)\big) 
            + \frac 12 \int_0^t \intO M(\Grad\varphi,\varphi)\, \abs{\Grad\mu}^2 \dx\ds
            \le E(\varphi_0)
        \end{align}
        for \revised{all} $t\in[0,T]$.
    \end{enumerate}
\end{definition}

\medskip

The existence of such a weak solution is ensured by the following theorem.

\begin{theorem} \label{THM:LOGPOT}
    Suppose that the assumptions \eqref{ass:dom}--\eqref{ass:M} are fulfilled, let $F$ be the logarithmic potential given by \eqref{DEF:F:LOG}, and let ${\varphi_0\in H^1(\Omega)}$ be any initial datum with $\abs{\varphi_0}\le 1$ a.e.~in $\Omega$ as well as $\abs{\mean{\varphi_0}}<1$. Then, there exists a weak solution $(\varphi,\mu)$ to the system \eqref{AICH} in the sense of Definition~\ref{DEF:WS*}. Moreover, there exists a constant $C>0$ depending only on $\varphi_0$, $E(\varphi_0)$, $c_0$ and the constants in \eqref{ass:dom}--\eqref{ass:M}, but not on $c_1$, such that for almost all $t\in [0,T]$,
    \begin{align}
    \label{EST:F':LOG}
        \bignorm{F'\big(\varphi(t)\big)}_{L^2(\Omega)}^2 \le C
        \Big[ 1 + \norm{F}_{L^\infty([-1,1])}^2 + 
            \norm{\Grad\mu(t)}_{L^2(\Omega)}^2\Big].
    \end{align}
\end{theorem}

The proof of this theorem can be found in Section~\ref{SECT:EXLOG}.

\begin{remark}
    The idea behind the proof of Theorem~\ref{THM:LOGPOT} is to construct a weak solution as the limit of a sequence of approximate solutions $(\varphi_n,\mu_n)$ where the logarithmic potential is approximated by a sequence $(G_n)_{n\in\N}$ of smooth potentials satisfying the assumption \eqref{ass:F2}. This means that Corollary~\ref{COR:REGPOT} can be applied to derive a suitable uniform bound on $G_n'(\varphi_n)$, which is crucial to pass to the limit in the weak formulation.

    We point out that this strategy could also be applied to construct weak solutions of the anisotropic Cahn--Hilliard equation \eqref{AICH} with regular potentials that do not satisfy the growth condition imposed in \eqref{ass:F1} or other types of singular potentials. For instance, the existence of a weak solution of system \eqref{AICH} with $F$ being the double-obstacle potential (see \eqref{DEF:F:DOB}) was established in \cite[Theorem~4.7]{GKNZ}.
\end{remark}

\subsection{Uniqueness of the weak solution in the case of constant mobility}

In case that the mobility function $M$ is positive and constant, we can even show that the weak solution given by Theorem~\ref{THM:REG} or Theorem~\ref{THM:LOGPOT}, respectively, is unique. However, if $F$ is a regular potential satisfying \eqref{ass:F1}, we need the following additional assumption:

\begin{enumerate}[label=$(\mathbf{F \arabic*})$, ref = $\mathbf{F \arabic*}$, start=3]
	\item \label{ass:F3} 
    The potential $F:\R\to\R$ can be decomposed as follows.
    There exist functions $F_1,F_2\in C^2(\R;\R)$ such that:
    \begin{enumerate}[label=\textnormal{(\roman*)}]
        \item $F(s) = F_1(s) + F_2(s)$ for all $s\in\R$.
        \item $F_1$ is convex and thus, $F_1'$ is monotonically increasing.
        \item $F_2'$ is Lipschitz continuous and thus, $F_2$ has at most quadratic growth.
    \end{enumerate}
\end{enumerate}

An analogous decomposition of the logarithmic potential is given by \eqref{DEC:LOG}. The uniqueness result reads as follows.

\begin{theorem} \label{THM:UNIQ}
    In addition to the assumptions \eqref{ass:dom}--\eqref{ass:M}, we assume that the mobility function $M$ is constant.
    Let $\varphi_0 \in H^1(\Omega)$ be arbitrary, and let $F$ be either a regular potential satisfying \eqref{ass:F1} and \eqref{ass:F3} or the logarithmic potential. 
    If $F$ is the logarithmic potential, we additionally assume that $\abs{\varphi_0}\le 1$ a.e.~in $\Omega$ as well as $\abs{\mean{\varphi_0}}<1$. 
    Then, the weak solution $(\varphi,\mu)$ given by Theorem~\ref{THM:REG} or Theorem~\ref{THM:LOGPOT}, respectively, is unique.
\end{theorem}

The proof of this theorem is given in Section~\ref{SECT:UNIQ}.

\subsection{Regularity of weak solutions}

By means of regularity theory for quasilinear elliptic PDEs, which will be discussed in Section~\ref{SECT:REG:QLE}, we are able to increase the spatial regularity of $\varphi$ under suitable additional regularity assumptions on $A$ and $\Omega$.

\begin{theorem} \label{THM:REG}
    In addition to the assumptions \eqref{ass:dom}--\eqref{ass:M}, we assume that $\Omega$ is of class $C^{1,1}$ and that $A\in C^{1,1}(\R^d)$.  
    Let $\varphi_0 \in H^1(\Omega)$ be arbitrary, and let $F$ be either a regular potential satisfying \eqref{ass:F1} or the logarithmic potential. 
    If $F$ is the logarithmic potential, we additionally assume that $\abs{\varphi_0}\le 1$ a.e.~in $\Omega$ as well as $\abs{\mean{\varphi_0}}<1$. 
    Moreover, let $(\varphi,\mu)$ be a corresponding weak solution in the sense of Definition~\ref{DEF:WS} or Definition~\ref{DEF:WS*}, respectively. Then, it additionally holds 
    \begin{equation}
    \label{REG:PHIA}
        \varphi \in L^2\big(0,T;H^2(\Omega)\big),
        \qquad 
        A'\big(\Grad\varphi(t)\big) \in H^1(\Omega) \quad\text{for almost all $t\in[0,T]$},
    \end{equation}
    and
    \begin{equation}
    \label{EQ:MU:PTW}
        \mu = -\Grad\cdot A'(\Grad\varphi) + F'(\varphi) 
        \quad\text{a.e.~in $\Omega_T$.}
    \end{equation}
\end{theorem}

\medskip

The proof is presented in Section~\ref{SECT:REG:AICH}.

If the mobility function $M$ is constant, we even obtain further regularity for $\delt\varphi$, $\mu$ and $F'(\varphi)$ as well as a \textit{strict separation property} for the phase-field at almost all times. 

\medskip

\begin{theorem} \label{THM:REG:2}
    In addition to the assumptions \eqref{ass:dom}--\eqref{ass:M}, we assume that $\Omega$ is of class $C^{1,1}$, that $A\in C^{1,1}(\R^d)$ and that the mobility function $M$ is constant. 
    Let $\varphi_0 \in H^1(\Omega)$ be arbitrary, and let $F$ be either a regular potential satisfying \eqref{ass:F1} with $p\le 4$ if $d=3$ and \eqref{ass:F3} or the logarithmic potential. 
    If $F$ is the logarithmic potential, we additionally assume that $\abs{\varphi_0}\le 1$ a.e.~in $\Omega$ as well as $\abs{\mean{\varphi_0}}<1$. 
    We further assume: \vspace{-1ex}
    \begin{enumerate}[label=$(\mathbf{I})$, ref = $\mathbf{I}$, ]
        \item \label{ASS:MU} There exists $\mu_0\in H^1(\Omega)$ such that for all $\eta\in H^1(\Omega)$, it holds:
        \begin{align*}
        \intO \mu_0 \eta \dx
        = \intO \ep A'(\Grad\varphi_0) \cdot \Grad \eta
            + \frac {1}{\ep} F'(\varphi_0) \eta \dx.
        \end{align*}
    \end{enumerate}
    Let $(\varphi,\mu)$ be the corresponding unique weak solution in the sense of Definition~\ref{DEF:WS} or Definition~\ref{DEF:WS*}, respectively.    
    Then, $(\varphi,\mu)$ has the following additional regularity: 
    \begin{align*}
        \varphi &\in H^1\big(0,T;H^1(\Omega)\big) \cap \revised{L^\infty\big(0,T;H^2(\Omega)\big)
        \cap C(\ov{\Omega_T})},
        \\
        \mu &\in L^\infty\big(0,T;H^1(\Omega)\big) \cap L^2\big(0,T;H^2(\Omega)\big),
        \\
        F'(\varphi) &\in L^\infty\big(0,T;L^2(\Omega)\big) \cap L^2\big(0,T;L^\infty(\Omega)\big).
    \end{align*}
    Moreover, if $F$ is the logarithmic potential, then for almost all $t\in[0,T]$, there exists $\delta(t) \in (0,1]$ such that
    \begin{equation}
    \label{EST:SEPPROP}
        \norm{\varphi(t)}_{L^\infty(\Omega)} \le 1 - \delta(t).
    \end{equation}
    This means that $\varphi(t)$ is strictly separated from $\pm 1$ for almost all times $t\in[0,T]$. 
    
    \revised{In the case $d=2$, the strict separation property even holds uniformly in time. This means that there exists $\delta^*>0$ such that for all $(t,\x)\in \Omega_T$,}
    \begin{equation}
    \label{EST:SEPPROP:UNI}
        \revised{\abs{\varphi(t,\x)} \le 1 - \delta^*\,.}
    \end{equation}
\end{theorem}

For the proof of this theorem, we also refer to Section~\ref{SECT:REG:AICH}.

\medskip

\begin{remark} \label{REM:REG:1}
    Suppose that the assumptions of Theorem~\ref{THM:REG:2} hold. \vspace{-1ex}
    \begin{enumerate}[label=\textnormal{(\alph*)}, ref = \textnormal{(\alph*)}]
        \item If the domain $\Omega$ is of class $C^{2,1}$,
        we even obtain $\mu \in L^2(0,T;H^3(\Omega))$ by means of classical elliptic regularity theory for the Poisson--Neumann problem (see, e.g., \cite[Section~4]{McLean}).
        \item Even without the additional regularity assumption \eqref{ASS:MU}, it would still be possible to obtain the regularities of Theorem~\ref{THM:REG} as well as statement~(a) of this remark, but with the interval $(0,T)$ in the function spaces replaced by $(t_0,T)$ for any time $t_0>0$. 
        \item \revised{We point out that assumption \eqref{ASS:MU} already implies that $\abs{\varphi_0} < 1$ holds a.e.~in $\Omega$.
        In the case $d=2$, it even holds $\norm{\varphi_0}_{L^\infty(\Omega)}<1$, i.e., $\varphi_0$ is strictly separated from $\pm 1$.
        This can be shown directly by proceeding similarly as in Step~4 and Step~5 of the proof of Theorem~\ref{THM:REG:2}. 
        Alternatively, the assertion for $d=2$ also follows from the uniform separation property \eqref{EST:SEPPROP:UNI}.}
    \end{enumerate}
\end{remark}


\section{Existence of weak solutions of the anisotropic Cahn--Hilliard equation with logarithmic potential} 
\label{SECT:EXLOG}

In this section, we prove the existence of a weak solution to the anisotropic Cahn--Hilliard equation \eqref{AICH}, where $F$ is chosen as the logarithmic potential (see \eqref{DEF:F:LOG}).

\begin{proof}[Proof of Theorem~\ref{THM:LOGPOT}]
    We split the proof into three steps.

    \textit{Step~1: Approximation of the logarithmic potential by smooth potentials.}
    The idea of this proof is to approximate the logarithmic potential $F$ by a suitable sequence $(G_n)_{n\in\N}$ of smooth potentials. 
    
    Therefore, for any $n\in\N$, we approximate the function 
    $$H: (0,\infty) \to \R,\quad H(s) = s\,\ln(s)$$
    by its second-order Taylor polynomial at the point $\frac 1n$, which is given by
    \begin{align*}
        H_n(s) &:= H\big(\tfrac 1n\big) + H'\big(\tfrac 1n\big)\left(s-\tfrac 1n\right) + \frac{1}{2} H''\big(\tfrac 1n\big)\left(s-\tfrac 1n\right)^2
        \\
        &\phantom{:}= - \tfrac{1}{2n} - s\ln(n) + \tfrac n2 s^2
    \end{align*}
    for all $s\in \R$. For all $s\in\R$ and $n\in\N$, we now define   
    \begin{align*}
		F_{1,n}(s) &:= \frac{\theta}{2}
		\begin{cases}
			(1-s)\ln(1-s) + H_n(1+s) &\text{if } s< -1+\frac{1}{n},\\
			(1-s)\ln(1-s) + (1+s)\ln(1+s) &\text{if } |s| \leq 1-\frac{1}{n},\\
			H_n(1-s) + (1+s)\ln(1+s) &\text{if } s> 1-\frac{1}{n},
		\end{cases}
        \\[1ex]
        G_n(s) &:= F_{1,n}(s) + F_2(s) = F_{1,n}(s) + \frac{\theta_c}{2}(1-s^2).
	\end{align*}
    It is straightforward to check that for all $n\in\N$, $F_{1,n} \in C^2(\R)$ with
    \begin{align*}
		F_{1,n}'(s) &= \frac{\theta}{2}
		\begin{cases}
			-\ln(1-s) + n(s+1) - 1 - \ln(n) &\quad\text{if } s< -1+\frac{1}{n},\\
			\ln(1+s) -\ln(1-s) &\quad\text{if } |s| \leq 1-\frac{1}{n},\\
			\ln(1+s) + n(s-1)  + 1 + \ln(n)  &\quad\text{if } s> 1-\frac{1}{n},
		\end{cases}
        \\[1ex]
        F_{1,n}''(s) &= \frac{\theta}{2}
		\begin{cases}
			(1-s)^{-1} + n &\quad\text{if } s< -1+\frac{1}{n},\\
			2(1-s^2)^{-1} &\quad\text{if } |s| \leq 1-\frac{1}{n},\\
			(1+s)^{-1} + n &\quad\text{if } s> 1-\frac{1}{n},
		\end{cases}
	\end{align*}
    for all $s\in\R$. For every $n\in\N$, we further introduce the approximate energy functional
    \begin{align}\label{DEF:EN}
        E_n:H^1(\Omega) \to \R,\quad
        E_n(\varphi) := \int_{\Omega} A(\nabla\varphi) 
            + G_n(\varphi) \dx.
    \end{align}   
    Defining 
    \begin{align}
        \label{DEF:N}
        N:= \min\left\{ m\in\N : m>\exp\left(\frac{2\theta_c}{\theta}\right) \right\},
    \end{align}
    we now show that the sequence $(G_n)_{n\in\N}$ has the following properties:%
    \begin{enumerate}[label=(\roman*)]
        \item For all $n\in\N$ and $s\in (-1,1)$, it holds $G_n(s) \le F(s)$.
        \item For all $n\in\N$ with $n \ge N$, $G_n$ satisfies \eqref{ass:F2} with $c_0 = \theta_c$ and $c_1 = \theta n - \theta_c$.
        \item $G_n \to F$ in $C^2_\mathrm{loc}((-1,1))$ as $n\to\infty$.
    \end{enumerate}
    
    To prove (i), we fix an arbitrary $n\in\N$. Let us first consider the case $s\in \big(1-\frac 1n,1\big)$.
    Recalling that $H''$ is monotonically decreasing, we use Taylor's formula to obtain
    \begin{align*}
        F(s) - G_n(s) 
        &= \tfrac{1}{2}\theta\big( H(1-s) - H_n(1-s) \big) 
        \\
        &= \tfrac{1}{4}\theta \big( H''(1-\xi) - H''(1-\tfrac 1n) \big) \big(s-\tfrac 1n\big)^2
        \ge 0
    \end{align*}
    for some $\xi \in \big[0,\tfrac 1n\big]$. Hence $F\ge G_n$ on $\big(1-\frac 1n,1\big)$. Proceeding similarly, we show that $F\ge G_n$ on the interval $\big(-1,-1+\frac 1n\big)$. As obviously $F=G_n$ on $\big[-1+\frac 1n,1-\frac 1n\big]$, this verifies (i).

    We further notice that $F_{1,n}''>0$ and thus, $F_{1,n}$ is strictly convex. Since $F_{1,n}$ is continuous, it attains its minimum on the compact interval $[-1,1]$. As $F_{1,n}'(0)=0$, we infer from the strict convexity of $F_{1,n}$ that this minimum is attained at $s=0$.
    We thus have
    \begin{align*}
        G_n(s) \ge F_{1,n}(s) \ge F_{1,n}(0) = 0 
        \quad\text{for all $s\in [-1,1]$}.
    \end{align*}
    For $N$ as defined in \eqref{DEF:N}, we now consider $n \ge N$ which implies $\tfrac 12 \theta n > \theta_c$. Then, for all $s\ge 1$, we have
    \begin{align*}
        G_n'(s) &= \tfrac 12 \theta \big[ \ln(1+s) + n(s-1) + 1 + \ln(n) \big] - \theta_c s
        \\
        & = \tfrac 12 \theta \big[ \ln(1+s) - n + 1 + \ln(n) \big] + \big[\tfrac 12 \theta n - \theta_c\big] s
        \\
        &\ge \tfrac 12 \theta \big[ \ln(n) - n \big] + \big[\tfrac 12 \theta n - \theta_c\big] 
        \\
        &= \tfrac 12 \theta \ln(n) - \theta_c \;> 0.
    \end{align*}
    In a similar fashion, we show that $G_n'(s) < 0$ for all $s\le 1$. In summary, we conclude that $G_n$ is non-negative for all $n \ge N$. 
    Moreover, it is straightforward to check that 
    \begin{align*}
        - \theta_c \le G_n''(s) \le \theta n - \theta_c
    \end{align*}
    for all $s\in \R$ and $n\in\N$. This means that (ii) is established.

    Let now $K$ be an arbitrary compact subset of $(-1,1)$. Then, there exists $n_K\in\N$ such that for all $n\in\N$ with $n\ge n_K$, it holds $K\subset \big(-1+\tfrac 1n,1-\tfrac 1n\big)$. This implies $G_n = F$ on $K$, which proves (iii).

    \textit{Step~2: A priori estimates for the sequence of approximate solutions.}
    For every $n\in\N$ with $n\ge N$, according to 
    Proposition~\ref{PROP:REGPOT}, there exists a weak solution $(\varphi_n,\mu_n)$ to the potential $G_n$ and the initial datum $\varphi_0$ in the sense of Definition~\ref{DEF:WS}. In the remainder of this proof, the letter $C$ will denote generic positive constants depending only on $F$, $\varphi_0$ and the quantities introduced in \eqref{ass:dom}--\eqref{ass:M}, but not on the approximation index $n$. 

    In the following, let $n\in\N$ with $n\ge N$ be arbitrary. We now recall that the weak solution $(\varphi_n,\mu_n)$ satisfies the weak energy dissipation law (cf.~Definition~\ref{DEF:WS}\ref{DEF:WS:ENERGY}) written for the energy functional $E_n$ that was introduced in \eqref{DEF:EN}. Invoking the properties (i) and (ii) from Step~1, we thus derive the estimate 
    \begin{align}
    \label{EST:EN}
        &\frac 12 \intO \abs{\Grad\varphi_n(t)}^2 \dx 
        + \frac 12 M_0 \int_0^t  \intO \abs{\Grad\mu_n(s)}^2 \dx\ds 
        \notag\\
        &\quad
        \le E_n\big(\varphi_n(t)\big) 
            + \frac 12  \int_0^t \intO M\big(\Grad\varphi_n(s),\varphi_n(s)\big) \abs{\Grad\mu_n(s)}^2 \dx\ds 
        \notag\\[1ex]
        &\quad 
        \le E_n(\varphi_0) \le E(\varphi_0) \le C
    \end{align}
    for almost all $t\in[0,T]$. Testing the weak formulation \eqref{DEF:WS:WF1} written for $(\varphi_n,\mu_n)$ with $\zeta\equiv 1$, we infer that $\mean{\varphi_n(t)} = \mean{\varphi_0}$ for almost all $t\in[0,T]$.
    Hence, employing Poincar\'e's inequality, we conclude
    \begin{align}
    \label{EST:PHIMUN}
        \norm{\varphi_n}_{L^\infty(0,T;H^1(\Omega))} + \norm{\Grad\mu_n}_{L^2(0,T;L^2(\Omega))} \le C.
    \end{align} 
    From the weak formulation \eqref{DEF:WS:WF2} written for $(\varphi_n,\mu_n)$, we directly infer
    \begin{align*}
        \norm{\delt\varphi_n}_{L^2(0,T;H^1(\Omega)')} \le C.
    \end{align*}
    Defining $\kappa := 1 - \abs{\mean{\varphi_0}} \in (0,1]$ and recalling that $G_n$ satisfies \eqref{ass:F2}, we use Corollary~\ref{COR:REGPOT} along with \eqref{EST:PHIMUN} and property (i) of Step~1 to obtain
    \begin{align}
    \label{EST:GNMU}
        \bignorm{G_n'\big(\varphi_n(t)\big)}_{L^2(\Omega)}^2
        \le \frac{c}{\kappa^2}
        \Big[ C + \norm{F}_{L^\infty([-1,1])}^2 + \norm{\Grad\mu(t)}_{L^2(\Omega)}^2 \Big].
    \end{align}
    Integrating this inequality with respect to time from $0$ to $T$, we deduce
    \begin{align}
    \label{EST:GN}
        \bignorm{G_n'\big(\varphi_n\big)}_{L^2(\OT)}^2 
        \le C
        \big( C + \norm{\Grad\mu}_{L^2(0,T;L^2(\Omega))}^2 \big)
        \le C
    \end{align}
    by means of \eqref{EST:PHIMUN}.
    After testing the weak formulation \eqref{DEF:WS:WF2} written for $(\varphi_n,\mu_n)$ with $\eta\equiv 1$ and taking the modulus of the resulting equation, we use H\"older's inequality to obtain
    \begin{align*}
        \left|\intO \mu_n(t) \dx\right| \le \intO \big| G_n'\big(\varphi_n(t)\big) \big| \dx \le C.
    \end{align*}
    Moreover, since $F_2'(s) = -\theta_c s$ for all $s\in\R$, we further have
    \begin{align*}
        \bignorm{F_{1,n}'\big(\varphi_n\big)}_{L^2(\OT)}
        \le \bignorm{G_n'\big(\varphi_n\big)}_{L^2(\OT)} + \theta_c \norm{\varphi_n}_{L^\infty(0,T;L^2(\Omega))}
        \le C.
    \end{align*}
    Overall, we have thus shown
    \begin{align}
    \label{EST:UNI}
        &\norm{\varphi_n}_{L^\infty(0,T;H^1(\Omega))} 
        + \norm{\delt\varphi_n}_{L^2(0,T;H^1(\Omega)')} 
        \notag\\
        &\quad + \norm{\mu_n}_{L^2(0,T;H^1(\Omega))} 
        + \bignorm{F_{1,n}'\big(\varphi_n\big)}_{L^2(\OT)}
        \le C.
    \end{align} 

    \textit{Step~3: Convergence to a weak solution.}
    In view of the uniform estimate \eqref{EST:UNI}, we now use the Banach--Alaoglu theorem as well as the Aubin--Lions--Simon lemma to conclude that there exist functions $\varphi$, $\mu$ and $f$ such that
    \begin{alignat}{2}
        \label{CONV:1}
        \delt \varphi_n &\to \delt \varphi
        &&\quad\text{weakly in $L^2(0,T;H^1(\Omega)')$,} 
        \\
        \label{CONV:2}
        \varphi_n &\to \varphi
        &&\quad\text{weakly-$^*$ in $L^\infty(0,T;H^1(\Omega))$,}
        \notag\\
        &&&\qquad\text{strongly in $C([0,T];L^2(\Omega))$, a.e.~in $\OT$,}
        \\
        \label{CONV:3}
        \mu_n &\to \mu
        &&\quad\text{weakly in $L^2(0,T;H^1(\Omega))$},
            \\
        \label{CONV:4}
        F_{1,n}'(\varphi_n) &\to f
        &&\quad\text{weakly in $L^2(0,T;L^2(\Omega))$},
    \end{alignat}
    as $n\to\infty$ along a non-relabeled subsequence. 
    As the functions $\big|F_{1,n}'(\varphi_n)\big|$ are measurable and non-negative, Fatou's lemma implies
    \begin{align}
    \label{EST:F1PN:FATOU}
        \int_{\OT} \underset{n\to\infty}{\lim\inf} \; \abs{F_{1,n}'(\varphi_n)}^2 \dtx
        \le \underset{n\to\infty}{\lim\inf} \; \int_{\OT} \abs{F_{1,n}'(\varphi_n)}^2 \dtx
        \le C.
    \end{align}
    We now fix arbitrary representatives of the functions $\varphi$ and $\varphi_n$, $n\in\N$, of their respective equivalence classes in order to make a pointwise evaluation possible.
    Due to \eqref{CONV:2}, there exists a null set $\mathcal M \subset \OT$ such that $\varphi_n\to\varphi$ pointwise on $\OT\setminus \mathcal M$. 
    We now define the sets
    \begin{align*}
        \mathcal M_- &:= \big\{ (t,\x) \in \OT\setminus \mathcal M \;\big\vert\; \varphi(t,\x) \le -1 \big\},
        \\
        \mathcal M_0\, &:= \big\{ (t,\x) \in \OT\setminus \mathcal M \;\big\vert\; \varphi(t,\x) \in (-1,1) \big\},
        \\
        \mathcal M_+ &:= \big\{ (t,\x) \in \OT\setminus \mathcal M \;\big\vert\; \varphi(t,\x) \ge 1 \big\},
    \end{align*}
    which entails that $\Omega\setminus \mathcal M = \mathcal M_- \cup \mathcal M_0 \cup \mathcal M_+$.

    Let now $(t,\x) \in \mathcal M_+$ be arbitrary. 
    As $\varphi_n(t,\x) \to \varphi(t,\x) \ge 1$ as $n\to\infty$, we have
    \begin{alignat}{2}
    \label{PHIN:1}
        \varphi_n(t,\x) > &\; 0 &&\quad\text{if $n$ is sufficiently large},
        \\
    \label{PHIN:2}
        \min\big\{\varphi_n(t,\x),1-\tfrac 1n\big\} \to &\; 1 &&\quad\text{as $n\to\infty$}.
    \end{alignat}
    Since $F_{1,n}'$ is increasing on $(0,\infty)$, we infer from \eqref{PHIN:1} that
    \begin{align}
    \label{EST:FNP}
        F_{1,n}'\big(\varphi_n(t,\x)\big) &\ge F_{1,n}'\big(\min\big\{\varphi_n(t,\x),1-\tfrac 1n\big\}\big)
        \notag\\
        &= F_{1}'\big(\min\big\{\varphi_n(t,\x),1-\tfrac 1n\big\}\big)
    \end{align}
    for all sufficiently large $n\in\N$.
    By the definition of $F_1$ and the convergence \eqref{PHIN:2}, we further deduce that the right-hand side of \eqref{EST:FNP} tends to infinity as $n\to\infty$. Hence, it directly follows that
    \begin{align}
    \label{DIV:FNP}
        F_{1,n}'\big(\varphi_n(t,\x)\big) \to \infty \quad\text{as $n\to\infty$}.
    \end{align}
    As $(t,\x) \in \mathcal M_+$ was arbitrary, this divergence holds for all $(t,\x) \in \mathcal M_+$.
    We thus conclude that the set $\mathcal M_+$ has Lebesgue measure zero, since otherwise this behavior would contradict estimate \eqref{EST:F1PN:FATOU}.

    Proceeding analogously, we show that $\mathcal M_-$ is a Lebesgue null set.

    In view of the decomposition $\Omega\setminus \mathcal M = \mathcal M_- \cup \mathcal M_0 \cup \mathcal M_+$, this proves that $\varphi \in (-1,1)$ in $\OT\setminus(\mathcal M\cup\mathcal M_+ \cup\mathcal M-)$ and thus almost everywhere in $\OT$.

    Let now $(t,\x) \in \OT\setminus(\mathcal M\cup\mathcal M_+ \cup\mathcal M_-) = \mathcal M_0$ be arbitrary. Since $\varphi_n(t,\x) \to \varphi(t,\x) \in (-1,1)$ as $n\to\infty$, there exists $n_0\in\N$ such that for all $n\ge n_0$,
    \begin{align*}
        -1 + \frac1n \le \varphi_n(t,\x) \le 1 - \frac 1n
    \end{align*}
    and consequently,
    \begin{align*}
        F_{1,n}\big(\varphi_n(t,\x)\big) &= F_{1}\big(\varphi_n(t,\x)\big) \to F_{1}\big(\varphi(t,\x)\big),
        \\
        F_{1,n}'\big(\varphi_n(t,\x)\big) &= F_{1}'\big(\varphi_n(t,\x)\big) \to F_{1}'\big(\varphi(t,\x)\big),
    \end{align*}
    as $n\to\infty$.
    This proves that 
    \begin{align}
        \label{CONV:F:PTW}
        F_{1,n}(\varphi_n) \to F_{1}(\varphi)
        \quad\text{and}\quad
        F_{1,n}'(\varphi_n) \to F_{1}'(\varphi)
        \quad\text{a.e.~in $\OT$},
    \end{align}
    as $n\to\infty$.
    In combination with the weak convergence \eqref{CONV:4}, we conclude $f = F_{1}'(\varphi)$ almost everywhere in $\OT$ and thus, 
    \begin{align}
    \label{CONV:F1P}
        F_{1,n}'(\varphi_n) \to F_{1}'(\varphi)
        \quad\text{weakly in $L^2(0,T;L^2(\Omega))$}.
    \end{align}    
    Moreover, since $F_2$ is quadratic $F_2'$ is linear, we use \eqref{CONV:2} to infer
    \begin{align}
    \label{CONV:F2}
        F_2(\varphi_n) \to F_2(\varphi)
        \quad\text{strongly in $C([0,T];L^1(\Omega))$},
        \\
    \label{CONV:F2P}
        F_2'(\varphi_n) \to F_2'(\varphi)
        \quad\text{strongly in $C([0,T];L^2(\Omega))$},
    \end{align}
    as $n\to\infty$.
    Combining \eqref{CONV:F1P} and \eqref{CONV:F2P}, we obtain
    \begin{align}
    \label{CONV:GNP}
        G_n'(\varphi_n) \to F'(\varphi) 
        \quad\text{weakly in $L^2(0,T;L^2(\Omega))$},
    \end{align}
    as $n\to\infty$.
    In particular, this means that all properties demanded by Definition~\ref{DEF:WS*}\ref{DEF:WS:REG*} are established.
    
    Using \eqref{EST:GNMU} and \eqref{CONV:GNP} along with the weak lower semi-continuity of the $L^2(\Omega)$-norm, we further obtain for almost all $t\in[0,T]$,
    \begin{align}
        \bignorm{F'\big(\varphi(t)\big)}_{L^2(\Omega)}^2
        &\le 
        \underset{n\to\infty}{\lim\inf} \;
        \bignorm{G_n'\big(\varphi_n(t)\big)}_{L^2(\Omega)}^2
        \notag\\
        &\le C
        \Big[ 1 + \norm{F}_{L^\infty([-1,1])}^2 + \norm{\Grad\mu(t)}_{L^2(\Omega)}^2 \Big].
    \end{align}
    This verifies estimate \eqref{EST:F':LOG}.

    Next, for any $n\in\N$, we use the weak formulation \eqref{DEF:WS:WF2} written for $(\varphi_n,\mu_n)$ to deduce 
    \begin{align}
        &a_0 \norm{\Grad\varphi_n - \Grad\varphi}_{L^2(\OT)}^2
        \le \int_{\OT} \big( A'(\Grad\varphi_n) - A'(\Grad\varphi) \big)
            \cdot \big( \Grad\varphi_n - \Grad\varphi \big) \;\mathrm d(t,\x)
        \notag\\
        &\quad= \int_{\OT} \mu_n\, (\varphi_n - \varphi) \;\mathrm d(t,\x)
            - \int_{\OT} G_n'(\varphi_n)\, (\varphi_n - \varphi) \;\mathrm d(t,\x)
        \notag\\
        &\qquad - \int_{\OT} A'(\Grad\varphi) \cdot \big( \Grad\varphi_n - \Grad\varphi \big) \;\mathrm d(t,\x).
    \end{align}
    Here, due to the convergences \eqref{CONV:2}, \eqref{CONV:3} and \eqref{CONV:GNP}, the right-hand side tends to zero as $n\to\infty$. This proves that
    \begin{align}
    \label{CONV:GPHI}
        \Grad \varphi_n \to \Grad\varphi
        \quad\text{strongly in $L^2(0,T;L^2(\Omega))$}.
    \end{align}
    In view of the growth conditions on $A$ and $A'$ from \eqref{ass:A}, Lebesgue's general convergence theorem (see, e.g., \cite[Section~3.25]{Alt}) implies
    \begin{alignat}{2}
        \label{CONV:A}
        A(\Grad\varphi_n) &\to A(\Grad\varphi)
        &&\quad\text{strongly in $L^1(\OT)$}, 
        \\
        \label{CONV:AP}
        A'(\Grad\varphi_n) &\to A'(\Grad\varphi)
        &&\quad\text{strongly in $L^2(\OT;\R^d)$}. 
    \end{alignat}
    Moreover, as the function $M$ is bounded (see \eqref{ass:M}), we use Lebesgue's dominated convergence theorem along with the weak-strong convergence principle to show that
    \begin{alignat}{2}
        \label{CONV:M}
        M(\Grad\varphi_n,\varphi_n) \Grad\mu_n &\to M(\Grad\varphi,\varphi) \Grad\mu
        &&\quad\text{weakly in $L^2(\OT;\R^d)$},
        \\
        \label{CONV:M:SQRT}
        \sqrt{M(\Grad\varphi_n,\varphi_n)} \Grad\mu_n &\to \sqrt{M(\Grad\varphi,\varphi)} \Grad\mu
        &&\quad\text{weakly in $L^2(\OT;\R^d)$}.
    \end{alignat}
    Using the convergences \eqref{CONV:1}--\eqref{CONV:3}, \eqref{CONV:GNP}, \eqref{CONV:AP} and \eqref{CONV:M}, we eventually pass to the limit in the weak formulation \eqref{DEF:WS:WF} written for $(\varphi_n,\mu_n)$ to show that $(\varphi,\mu)$ satisfies the weak formulation \eqref{DEF:WS:WF*}.
    Moreover, since $\varphi_n(0) = \varphi_0$ a.e.~in $\Omega$, the strong convergence stated in \eqref{CONV:2} directly implies $\varphi(0) = \varphi_0$ a.e.~in $\Omega$. This means that all conditions of Definition~\ref{DEF:WS*}\ref{DEF:WS:WEAK*} are verified.

    It remains to prove the weak energy dissipation law. Let $\sigma\in C^\infty([0,T];[0,\infty))$ be an arbitrary test function.
    Using the strong convergences \eqref{CONV:A} and \eqref{CONV:F2} as well as the pointwise-a.e.~convergence \eqref{CONV:F:PTW} along with Fatou's lemma, we infer
    \begin{align}
    \label{LIMINF:E}
        &\int_0^T E\big(\varphi(t)\big) \, \sigma(t) \dt 
        = \int_0^T \intO \big[ A\big(\Grad\varphi(t)\big) + F_1\big(\varphi(t)\big) + F_2\big(\varphi(t)\big) \big]\, \sigma(t) \dx \dt
        \notag\\
        &\quad \le \underset{n\to\infty}{\lim\inf} \;
        \int_0^T \intO \big[ A\big(\Grad\varphi_n(t)\big) + F_{1,n}\big(\varphi_n(t)\big) + F_2\big(\varphi_n(t)\big) \big]\, \sigma(t) \dt
        \notag\\
        &\quad = \underset{n\to\infty}{\lim\inf} \; 
        \int_0^T E_n\big(\varphi_n(t)\big) \, \sigma(t) \dt .
    \end{align}
    In view of the weak convergence \eqref{CONV:M:SQRT}, we use the weak lower semi-continuity of the $L^2(\OT)$-norm as well as Fatou's lemma to deduce
    \begin{align}
    \label{LIMINF:M}
        &\int_0^T \frac 12 \int_0^t \intO M(\Grad\varphi,\varphi) \abs{\Grad\mu}^2 \dx \ds \, \sigma(t) \dt
        \notag\\
        &\le \underset{n\to\infty}{\lim\inf} \; \int_0^T \frac 12 \int_0^t \intO M(\Grad\varphi_n,\varphi_n) \abs{\Grad\mu_n}^2 \dx\ds \, \sigma(t) \dt.
    \end{align}
    We further recall that $(\varphi_n,\mu_n)$ satisfies the weak energy dissipation law and that $E_n(\varphi_0) \le E(\varphi_0)$ for all $n\in\N$ due to property (i) of Step~1. Hence, combining \eqref{LIMINF:E} and \eqref{LIMINF:M}, we conclude
    \begin{align}
        &\int_0^T \left( E\big(\varphi(t)\big) 
            + \frac 12 \int_0^t \intO M(\Grad\varphi,\varphi) \abs{\Grad\mu}^2 \dx \ds \right) \,\sigma(t) \dt 
        \notag\\
        & \le \underset{n\to\infty}{\liminf} \; 
        \Bigg[ \int_0^T \left(  E_n\big(\varphi_n(t)\big) 
            + \frac 12 \int_0^t \intO M(\Grad\varphi_n,\varphi_n) \abs{\Grad\mu_n}^2 \dx \ds \right) \,\sigma(t) \dt \Bigg]
        \notag\\
        & \le \underset{n\to\infty}{\liminf} \; \int_0^T E_n\big(\varphi_0 \big) \, \sigma(t) \dt 
        \;\le\;  \int_0^T E(\varphi_0) \, \sigma(t) \dt.
    \end{align}
    As the test function $\sigma$ was arbitrary, this verifies the energy dissipation law \eqref{DEF:WS:DISS*} \revised{for almost all $t\in[0,T]$. Obviously, the time integral in \eqref{DEF:WS:DISS*} is non-negative and continuous with respect to $t\in[0,T]$. Moreover, the functions 
    \begin{equation*}
        t\mapsto \intO A(\Grad \varphi(t)) \dx
        \quad\text{and}\quad
        t\mapsto \intO F(\varphi(t)) \dx
    \end{equation*}
    are lower semi-continuous on $[0,T]$. For the first mapping, this holds since $A$ is convex and continuous,
    and for the second one, it follows by means of Fatou's lemma. Consequently, the mapping $t\mapsto E(\varphi(t))$ is lower semi-continuous on $[0,T]$. This entails that the energy dissipation law \eqref{DEF:WS:DISS*} actually holds true for \textit{all} times $t\in[0,T]$.
    }
    
    This means that Definition~\ref{DEF:WS*}\ref{DEF:WS:ENERGY*} is verified.

    In summary, we have shown that the pair $(\varphi,\mu)$ is a weak solution in the sense of Definition~\ref{DEF:WS*} which additionally satisfies the estimate \eqref{EST:F':LOG}. Thus, the proof is complete.
\end{proof}


\section{Uniqueness of the weak solution} \label{SECT:UNIQ}

In this section, we present the proof of Theorem~\ref{THM:UNIQ}, which states (under suitable assumptions on the potential $F$) that weak solutions of the anisotropic Cahn--Hilliard equation are unique provided that the mobility function $M$ is constant.

\begin{proof}[Proof of Theorem~\ref{THM:UNIQ}]
As the mobility function $M$ is assumed to be constant, we simply set $M\equiv 1$ without loss of generality.

Let $(\varphi,\mu)$ be a weak solution given by Proposition~\ref{PROP:REGPOT} if $F$ is a regular potential or given by Theorem~\ref{THM:LOGPOT} if $F$ is the logarithmic potential.
We now assume that $(\varphi_*,\mu_*)$ is a further weak solution of the anisotropic Cahn--Hilliard equation \eqref{AICH} to the same initial datum $\varphi_0$ in the sense of Definition~\ref{DEF:WS} if $F$ is a regular potential or in the sense of Definition~\ref{DEF:WS*} if $F$ is the logarithmic potential. In the following, we use the notation
\begin{align*}
    \ov\varphi := \varphi - \varphi_* \quad\text{and}\quad \ov\mu := \mu - \mu_*.
\end{align*}
We point out that 
\begin{align*}
    \ov\varphi(0) = 0 
    \qquad\text{and}\qquad
    \ov\varphi(t) \in H^1_{(0)}(\Omega) \;\;\text{for almost all $t\in [0,T]$}.
\end{align*}
In order to prove uniqueness, we intend to show that $\norm{\ov\varphi(t)}_{-1} = 0$ for all $t\in [0,T]$.
\revised{From the weak formulations written for $(\varphi,\mu)$ and $(\varphi_*,\mu_*)$, we deduce
\begin{subequations}
\label{WF:DIFF}
\begin{align}
    \label{WF:DIFF:1}
    \langle \delt \ov\varphi(s) , \zeta \rangle 
    &= - \intO \Grad \ov\mu(s) \cdot \Grad \zeta \dx,
    \\
    \label{WF:DIFF:2}
    \intO \ov\mu(s) \, \eta \dsx
    &= \intO \big[ A'\big(\Grad\varphi(s)\big)- A'\big(\Grad\varphi_*(s)\big) \big] \cdot \Grad\eta \dx
    \notag\\
    &\qquad
    + \intO \big[F'\big(\varphi(s)\big) - F'\big(\varphi_*(s)\big) \big] \eta \dx
\end{align}
\end{subequations}
for all test functions $\zeta,\eta\in H^1(\Omega)$ and almost all $s\in [0,T]$.
Testing \eqref{WF:DIFF:1} with $\zeta \equiv 1$, we immediately observe that $\delt\ov\varphi(s) \in H^{-1}_{(0)}(\Omega)$ for almost all $s\in [0,T]$.

Let now $t \in [0,T]$ be arbitrary. In the following, we use the notation $\Ot := (0,t) \times \Omega$. 
For almost all $s\in[0,t]$, we now test \eqref{WF:DIFF:2} with $\eta:= \ov\varphi(s) \in L^2(0,T;H^1(\Omega))$ and integrate the resulting equation with respect to $s$ from $0$ to $t$. Recalling the monotonicity of $A'$ and $F_1'$ as well as the Lipschitz continuity of $F_2'$, we obtain 
\begin{align}
    \label{DIFF:1*}
    &\int_{\Ot} \ov\mu \, \ov\varphi \dsx 
    \notag\\
    &\quad = \int_{\Ot} (A'(\Grad\varphi) -
        A'(\Grad\varphi_*))  \cdot \Grad\ov\varphi \dsx
    \notag\\
    &\qquad
        + \int_{\Ot}
        \big[F_1'(\varphi) - F_1'(\varphi_*) \big] \ov\varphi
        + \big[F_2'(\varphi) - F_2'(\varphi_*) \big] \ov\varphi
        \dsx
    \notag\\
    &\quad \ge a_0 \int_{\Ot} \abs{\Grad\ov\varphi}^2 \dsx
        - L \int_{\Ot} \abs{\ov\varphi}^2
        \dsx,
\end{align}
where $L$ denotes the minimal Lipschitz constant of $F_2'$. 
Invoking the definition of $\PN$ (see Subsection~\ref{SUBSECT:NOT}),
we derive the inequality
\begin{align*}
    \norm{\ov\varphi(s)}_{L^2(\Omega)}^2 
    &= \intO \Grad \PN \ov \varphi(s) \cdot \Grad \ov\varphi(s) \dx
    \notag\\
    &
    \le \frac{a_0}{2} \norm{\Grad\ov\varphi(s)}_{L^2(\Omega)}^2
        + \frac{1}{2a_0} \norm{\ov\varphi(s)}_{-1}^2
\end{align*}
for almost all $s\in [0,T]$. Using this estimate to bound the last line of \eqref{DIFF:1*} from below, we conclude
\begin{align}
    \label{DIFF:1}
    \int_{\Ot} \ov\mu \, \ov\varphi \dsx 
    \ge \frac{a_0}{2} \int_0^t \norm{\Grad\ov\varphi(s)}_{L^2(\Omega)}^2 \ds    
        - C \int_0^t \frac 12 \norm{\ov\varphi(s)}_{-1}^2 \ds,
\end{align}
where $C$ is a positive constant depending on $L$ and $a_0$.
By proceeding exactly as for the isotropic Cahn--Hilliard equation,
the integral on the right-hand side can be expressed as
\begin{align}
    \label{DIFF:2}
    \int_{\Ot} \ov\varphi \, \ov\mu \dsx = - \frac 12 \norm{\ov\varphi(t)}_{-1}^2.
\end{align}
(For more details about the derivation of this identity, we refer to \cite{Garcke2003} or \cite{Garcke2020}.)
Combining \eqref{DIFF:1} and \eqref{DIFF:2}, we thus arrive at}
\begin{align*}
    \revised{\frac 12 \norm{\ov\varphi(t)}_{-1}^2} + \frac{a_0}{2} \int_0^t \norm{\Grad\ov\varphi(s)}_{L^2(\Omega)}^2 \ds 
    \le C \int_0^t \frac 12 \norm{\ov\varphi(s)}_{-1}^2 \ds
\end{align*}
for all $t\in [0,T]$. Hence, applying Gronwall's lemma, we eventually conclude 
\begin{align*}
    \norm{\ov\varphi(t)}_{-1} = 0 \quad\text{for all $t\in[0,T]$}.
\end{align*}
This directly implies $\ov\varphi = 0$ a.e.~in $\Omega_T$. As a consequence, the right-hand side of \eqref{WF:DIFF:2} vanishes and hence, we also have $\ov\mu = 0$ a.e.~in $\Omega_T$. 

This proves that $(\varphi,\mu)$ is the unique weak solution and thus, the proof is complete.
\end{proof}


\section{Regularity theory for a class of quasilinear elliptic PDEs} \label{SECT:REG:QLE}

In this section, we present regularity results for a class of quasilinear elliptic PDEs. They will be applied in the proof of Theorem~\ref{THM:REG} in order to improve the spatial regularity of the phase-field component of a weak solution to the anisotropic Cahn--Hilliard equation \eqref{AICH}.

We make the following assumptions.
\begin{enumerate}[label=$(\mathbf{R \arabic*})$, ref = $\mathbf{R \arabic*}$]
	\item \label{reg:ass:dom} The set $U\subset\R^n$ with $n\in\N$ is a bounded open subset of $\R^n$ with Lipschitz boundary and outer unit normal vector field $\n$.
    \item \label{reg:ass:a} The function $a:\R^n\times\R^n \to \R^n$, $(\x,\p) \mapsto a(\x,\p)$ is 
	\begin{enumerate}[label=$(\mathbf{R \arabic{enumi}.\arabic*})$, ref = $\mathbf{R \arabic{enumi}.\arabic*}$]
        \item \label{reg:ass:a:Lip} Lipschitz continuous in $\p$, i.e., there exists a constant $C_L\geq0$ such that
        \begin{equation*}
            \bigabs{a(\x,\p)-a(\x,\q)}
            \leq C_L \abs{\p-\q}
		\end{equation*}
        for all $\p,\q\in\R^n$ and $\x\in\R^n$,
		\item \label{reg:ass:a:mon} strongly monotone in $\p$, i.e., there exists a constant $C_M>0$ such that
		\begin{equation*}
            \big(a(\x,\p)-a(\x,\q)\big) \cdot (\p-\q)
            \geq C_M \abs{\p-\q}^2
		\end{equation*}
        for all $\p,\q\in\R^n$ and $\x\in\R^n$,
		\item \label{reg:ass:a:qLip} quasi-Lipschitz continuous in $\x$, i.e., there exists a constant $C_Q\geq0$ such that
		\begin{equation*}
            \bigabs{a(\x,\p)-a(\y,\p)}
            \leq C_Q \big( \abs{\p}+1 \big) \abs{\x-\y}
		\end{equation*}
        for all $\x,\y\in\R^n$ and $\p\in\R^n$.
	\end{enumerate}
    \item \label{reg:ass:f} The function $f:U\to\R$ belongs to $L^2(U)$.
\end{enumerate}

We now establish a regularity theory for weak solutions of the following quasi-linear elliptic PDE:
\begin{subequations}
\label{reg:QLE} 
\begin{alignat}{2}
    \label{reg:QLE:1}
    -\Grad \cdot a(\cdot,\Grad u) &= f
    &&\quad\text{in $U$},\\
    \label{reg:QLE:2}
    a(\cdot,\Grad u) \cdot \n &= 0
    &&\quad\text{on $\partial U$}.
\end{alignat}
\end{subequations}
A weak solution of this boundary value problem is defined as follows.

\begin{definition} \label{reg:DEF:WS}
    Under the assumptions \eqref{reg:ass:dom}--\eqref{reg:ass:f}, a function $u\in H^1(U)$ is called a weak solution of the boundary value problem \eqref{reg:QLE}
    if $u$ satisfies the weak formulation
    \begin{align} 
        \label{reg:QLE:WF}
        \int_U a(\cdot,\Grad u) \cdot \Grad v \dx
        = \int_U f v \dx
        \quad \text{for all $v\in H^1(U)$}.
    \end{align}
\end{definition}

To prove additional regularity of such weak solutions, we first recall the notion of difference quotients.
Let $U\subset\R^n$ with $n\in\N$ be a bounded open subset of $\R^n$. For any locally integrable function $w\colon U\to\R$ and any open subset $V\ssubset U$, for $k\in\{1,\dots,d\}$, the $k$-th difference quotient of step size $h$ is defined by 
\begin{align*}
	\Grad_k^h w(\x)
	:= \frac{w(\x+h\e_k) - w(\x)}{h}
\end{align*}
for $\x\in V$ and $h\in\R$ with $0<|h|<\dist(V,\partial U)$. Here, $\e_k\in\R^n$ denotes the $k$-th unit vector in $\R^n$. For vector-valued functions the difference quotient is to be understood componentwise.

\subsection{Regularity in the interior}

\begin{theorem} \label{reg:THRM:reg:int}
    Suppose that the assumptions \eqref{reg:ass:dom}--\eqref{reg:ass:f} are fulfilled and let $u\in H^1(U)$ be a weak solution of \eqref{reg:QLE} in the sense of Definition~\ref{reg:DEF:WS}. Then, it holds
    \begin{align*}
        u\in H^2_\mathrm{loc}(U),
    \end{align*}
    and for every open subset $V\ssubset U$, there exists a constant $C\ge0$ depending only on $U$, $V$ and $a$ such that the estimate 
    \begin{align}
        \label{reg:EST:u:H2V}
        \norm{u}_{H^2(V)}
        \le C \big( \norm{f}_{L^2(U)} + \norm{u}_{L^2(U)} + 1 \big)
    \end{align}
    holds. 
\end{theorem}

To prove Theorem~\ref{reg:THRM:reg:int}, we recall the following properties of difference quotients.

\begin{lemma} \label{reg:LMM:DQ}
    Let $U\subset\R^n$ with $n\in\N$ be a bounded open subset of $\R^n$ and let $V\ssubset U$ be any open subset. Then, it holds:
    \begin{enumerate}[label=\textnormal{(\alph*)}]
		\item \label{reg:DQ.a}
		If $w_1,w_2\in L^2(U)$ are functions with $\supp w_i\subset V$ for $i=1$ or $i=2$, then the identity 
		\begin{align*}
			\int_U w_1\Grad_k^{-h}w_2 \dx
			= -\int_U \Grad_k^{h}w_1 w_2 \dx
		\end{align*}
        holds for all $0<\abs{h}< \dist(V,\partial U)$ and all $k\in\oton$.
		\item\label{reg:DQ.b}
		Suppose $w\in H^1(U)$ and $k\in\oton$. Then, for every open set $U'\ssubset U$, the difference quotient $\dkh w$ belongs to $L^2(U')$ for all $0<\abs{h}< \dist(U',\partial U)$, and it satisfies the estimate
		\begin{align*}
			\bignorm{\dkh w}_{L^2(U')}
			\leq \norm{\del_k w}_{L^2(U)}.
		\end{align*}
		\item \label{reg:DQ.c}
		Let $w\colon U\to\R^n$ be a locally integrable function with $w\in L^2(V)$ and let $k\in\oton$ be arbitrary. Moreover, assume that there exists a constant $C\ge 0$ such that $\dkh w \in L^2(V)$ and $\norm{\dkh w}_{L^2(V)} \le C$ for all $0<\abs h< \dist(V,\partial U)$. Then, the weak partial derivative $\del_k w$ exists and it holds
		\begin{align*}
			\del_k w\in L^2(V) \quad\text{with}\quad
			\norm{\del_k w}_{L^2(V)} \le C.
		\end{align*}
	\end{enumerate}
\end{lemma}

\begin{proof}
    The result \ref{reg:DQ.a} follows directly by a change of variables. A proof of \ref{reg:DQ.b} can be found, e.g., in \cite[Lemma~7.23]{Gilbarg2001}. Moreover, \ref{reg:DQ.c} can be established as in \cite[§5.8.2~Theorem~3]{Evans2010}.
\end{proof}

We now proceed with the proof of Theorem~\ref{reg:THRM:reg:int}.

\begin{proof}[Proof of Theorem~\ref{reg:THRM:reg:int}.]
    The idea of this proof is to use Lemma~\ref{reg:LMM:DQ}\ref{reg:DQ.c} and therefore, to bound every difference quotient of $\Grad u$ uniformly in $L^2(V)$. 

    Let $V\ssubset U$ be an arbitrary open subset with $\delta:=\dist(V,\partial U)>0$. Defining 
    $$W':=\bigcup_{\x\in V} B_{\delta/4}(\x) \quad\text{and}\quad W:=\bigcup_{\x\in V} B_{\delta/2}(\x),$$
    the subsets $W', W\subset U$ are open with $V\ssubset W'\ssubset W\ssubset U$. 

    In this proof, unless stated otherwise, $C$, $C_1$, $C_2$, $C_3$ will denote non-negative constants depending only on $U$, $V$ and $a$. The value of $C$ may change from line to line. We point out that due to their construction, the subsets $W'$ and $W$ depend only on $U$ and $V$. 

    \textit{Step~1.} To ensure a positive distance to the boundary $\partial U$, we fix a cutoff function $\zeta\in C^\infty(\R^n)$ with $0\le\zeta\le1$ such that $\zeta=1$ in $V$ and $\zeta=0$ in $\R^n\setminus W'$. In particular, by this construction, the function $\zeta$ depends only on $U$ and $V$.
    We write $w:=\dkh u$ to denote the $k$-th difference quotient of $u$ and we test the weak formulation \eqref{reg:QLE:WF} with the second-order difference quotient $v:=-\dkmh (\zeta^2w)$, where $k\in\oton$ is arbitrary and $0<\abs h< \frac{\delta}{4}$. On the left-hand side, using Lemma~\ref{reg:LMM:DQ}\ref{reg:DQ.a}, we obtain
	\begin{align}
        \label{EST:R1}
		&\int_U a(\cdot,\Grad u) \cdot \Grad v \dx
		= -\int_U a(\cdot,\Grad u) \cdot \dkmh \big(\Grad (\zeta^2w)\big) \dx\nonumber\\
		&\quad= \int_U \dkh \big(a(\cdot,\Grad u)\big) \cdot \Grad (\zeta^2w) \dx\nonumber\\
		&\quad= \int_U \dkh \big(a(\cdot,\Grad u)\big) \cdot 2\zeta \Grad\zeta w \dx
        +  \int_U \dkh \big(a(\cdot,\Grad u)\big) \cdot \zeta^2\Grad w \dx.
	\end{align}
    Employing the Lipschitz continuity of $a$ with respect to $\p$ from \eqref{reg:ass:a:Lip}, the quasi Lipschitz continuity of $a$ with respect to  $\x$ from \eqref{reg:ass:a:qLip}, Lemma~\ref{reg:LMM:DQ}\ref{reg:DQ.b} as well as Young's inequality, the modulus of the first summand on the right-hand side of \eqref{EST:R1} can be bounded by
    \begin{align}
	\label{reg:lhs:1.summand}
		&\biggabs{\int_U \dkh \big(a\big(\x,\Grad u(\x)\big)\big) \cdot 2\zeta \Grad\zeta w \dx} \nonumber\\
		&\le \int_W \bigg( \frac{1}{|h|} \Bigabs{a\big(\x+h\e_k, \Grad u(\x+h\e_k)\big) - a\big(\x+h\e_k, \Grad u(\x)\big) } \nonumber\\
		&\quad + \frac{1}{|h|} \Bigabs{a\big(\x+h\e_k, \Grad u(\x)\big) - a\big(\x, \Grad u(\x)\big) } \bigg) 2\zeta \, |\Grad\zeta| \, |w| \dx \nonumber\\
        &\le \int_W \bigg( \frac{C_L}{|h|}  \bigabs{ \Grad u(\x+h\e_k)  - \Grad u(\x) } 
        +  C_Q \big(|\Grad u(\x)|+1\big) 
			\bigg) 2\zeta \, |\Grad\zeta| \, |w| \dx \nonumber\\
		&= \int_W \Big(C_L |\Grad w| + C_Q \big(|\Grad u|+1\big) \Big) 2\zeta \, |\Grad\zeta| \, |w| \dx \nonumber\\
		&\le \frac{C_M}{4} \int_W \zeta^2 |\Grad w|^2 \dx 
        + C \int_W |\Grad u|^2 + 1 \dx 
		+ C \int_W |\Grad\zeta|^2 w^2 \dx \nonumber\\
		&\le \frac{C_M}{4} \int_W \zeta^2 |\Grad w|^2 \dx 
		+ C \int_W |\Grad u|^2 + 1 \dx 
		+ C \int_{W'} w^2 \dx \nonumber\\
		&\le \frac{C_M}{4} \int_W \zeta^2 |\Grad w|^2 \dx 
		+ C \int_W |\Grad u|^2 + 1 \dx 
		+ C \int_W |\Grad u|^2 \dx \nonumber\\
		&\le \frac{C_M}{4} \int_W \zeta^2 |\Grad w|^2 \dx 
		+ C_1 \int_W |\Grad u|^2 \dx 
		+ C_1.
	\end{align}
    Invoking the strong monotonicity of $a$ with respect to $\p$ (see \eqref{reg:ass:a:mon}) and the quasi Lipschitz continuity of $a$ with respect to $\x$ (see \eqref{reg:ass:a:qLip}), we use Young's inequality to deduce that the second summand on the right-hand side of \eqref{EST:R1} can be estimated by
    \begin{align}
	\label{reg:lhs:2.summand}
		&\int_U \dkh \big(a \big(\x,\Grad u(\x)\big)\big) \cdot \zeta^2\Grad w \dx \nonumber\\
		&\ge \int_W \frac{1}{h} \Big( a\big(\x+h\e_k, \Grad u(\x+h\e_k)\big) - a\big(\x+h\e_k, \Grad u(\x)\big) \Big) \cdot \zeta^2\Grad w \dx \nonumber\\
		&\quad- \int_W \frac{1}{|h|} \Bigabs{ a\big(\x+h\e_k, \Grad u(\x)\big) - a\big(\x, \Grad u(\x)\big) } \, \zeta^2 \, |\Grad w| \dx \nonumber\\
		&\ge \int_W \frac{1}{h^2} C_M \bigabs{ \Grad u(\x+h\e_k) - \Grad u(\x) }^2  \zeta^2 \dx \nonumber\\
		&\quad - \int_W \frac{1}{|h|} C_Q \big(|\Grad u(\x)|+1 \big) |h\e_k| \, \zeta^2 \, |\Grad w| \dx \nonumber\\
		&= \int_W C_M \zeta^2|\Grad w|^2 \dx
		- \int_W C_Q \big(|\Grad u|+1\big) \zeta^2 |\Grad w| \dx \nonumber\\
		&\ge \int_W C_M \zeta^2 |\Grad w|^2 \dx
		- \frac{C_M}{4} \int_W \zeta^2 |\Grad w|^2 \dx
		- C \int_W |\Grad u|^2+1 \dx \nonumber\\
		&\ge \frac{3C_M}{4} \int_W \zeta^2|\Grad w|^2 \dx
		- C_2 \int_W |\Grad u|^2 \dx - C_2.
    \end{align}
    Furthermore, applying Lemma~\ref{reg:LMM:DQ}\ref{reg:DQ.b} twice, we derive the estimate
	\begin{align*}
		&\int_W |v|^2 \dx
		= \int_{W'} \bigabs{\dkmh (\zeta^2w)}^2 \dx
        \\
		&\quad 
		\le C \int_W \bigabs{\Grad (\zeta^2w)}^2 \dx
        = C \int_{W'} \bigabs{2\zeta \Grad\zeta w + \zeta^2 \Grad w}^2 \dx
        \\
		&\quad 
		\le C \int_{W'} |w|^2 + \zeta^2 |\Grad w|^2 \dx
		\le C_3 \int_W |\Grad u|^2 + \zeta^2 |\Grad w|^2 \dx.
	\end{align*}
    By means of Young's inequality, the right-hand side of \eqref{reg:QLE:WF} can be bounded by
	\begin{align}
	\label{reg:rhs}
	   \biggabs{ \int_U fv \dx}
		&\le \frac{C_M}{4C_3} \int_W |v|^2 \dx
			+ C \int_W f^2 \dx \nonumber\\
		&\le \frac{C_M}{4} \int_W |\Grad u|^2 + \zeta^2 |\Grad w|^2 \dx
			+ C \int_W f^2 \dx \nonumber\\
		&\le \frac{C_M}{4} \int_W \zeta^2 |\Grad w|^2 \dx
			+ C \int_W f^2 + |\Grad u|^2 \dx.
		\end{align}
    Altogether, combining the estimates \eqref{reg:lhs:1.summand}, \eqref{reg:lhs:2.summand}, \eqref{reg:rhs} and the weak formulation \eqref{reg:QLE:WF}, we conclude
	\begin{align*}
		&\frac{3C_M}{4} \int_W \zeta^2|\Grad w|^2 \dx
			- C_2 \int_W |\Grad u|^2 \dx - C_2 \\
		&\qquad- \bigg( \frac{C_M}{4} \int_W \zeta^2 |\Grad w|^2 \dx 
			+ C_1 \int_W |\Grad u|^2 \dx 
			+ C_1 \bigg) \\
		&\quad\le \biggabs{ \int_U a(\cdot,\Grad u) \cdot \Grad v \dx }
		= \biggabs{ \int_U fv \dx }\\
		&\quad\le \frac{C_M}{4} \int_W \zeta^2 |\Grad w|^2 \dx
			+ C \int_W f^2 + |\Grad u|^2 \dx,
	\end{align*}
    which eventually yields
    \begin{align*}
		\int_V \bigabs{\dkh\Grad u}^2 \dx
		&= \int_V |\Grad w|^2 \dx
		\le \int_W \zeta^2 |\Grad w|^2 \dx
        \notag\\
		&\le C \int_W f^2 + |\Grad u|^2 \dx + C
	\end{align*}
    for all $k\in\oton$ and $0<|h|< \frac{\delta}{4}$. Therefore, every difference quotient of $\Grad u$ is bounded uniformly in $h$ as desired. Lemma~\ref{reg:LMM:DQ}\ref{reg:DQ.c} now implies $\partial_k \Grad u\in L^2(V)$ for all $k\in\oton$ and thus $D^2 u\in L^2(V)$ with
	\begin{align*}
		\norm{D^2 u}_{L^2(V)}
		\le C\big(\norm{f}_{L^2(W)} +\norm{\Grad u}_{L^2(W)} + 1 \big).
	\end{align*}
    This shows $u\in H^2(V)$ along with the estimate
	\begin{align}
	\label{reg:est:with:Gradu}
		\norm{u}_{H^2(V)}^2
		&= \norm{D^2 u}_{L^2(V)}^2 + \norm{\Grad u}_{L^2(V)}^2 + \norm{u}_{L^2(V)}^2 \nonumber\\
		&\le C\big(\norm{f}_{L^2(W)}^2 + \norm{\Grad u}_{L^2(W)}^2 + \norm{u}_{L^2(W)}^2 +1\big).
	\end{align}

    \textit{Step~2.} Our next goal is to further estimate the term $\norm{\Grad u}_{L^2(W)}^2$ appearing on the right-hand side of \eqref{reg:est:with:Gradu} in order to verify \eqref{reg:EST:u:H2V}. For this purpose, let $\zeta\in C^\infty(\R^n)$ be a new cutoff function with $0\le\zeta\le1$ such that $\zeta=1$ in $W$ and $\supp\zeta \subset U$. The weak formulation \eqref{reg:QLE:WF} is now tested with $v=\zeta^2u$ and we proceed as in Step~1. First, we compute the left-hand side of the resulting equation as
	\begin{align}
    \label{EST:a:1}
		\int_U a(\cdot,\Grad u) \cdot \Grad (\zeta^2u) \dx
		= \int_U a(\cdot,\Grad u) \cdot 2\zeta \Grad\zeta u \dx
			+ \int_U a(\cdot,\Grad u) \cdot \zeta^2\Grad u \dx.
	\end{align}
    As in \eqref{reg:lhs:1.summand}, invoking the assumptions \eqref{reg:ass:a:Lip} and \eqref{reg:ass:a:qLip}, we estimate the modulus of the first summand on the right-hand side of \eqref{EST:a:1} by
	\begin{align*}
		\biggabs{ \int_U a(\cdot,\Grad u) \cdot 2\zeta \Grad\zeta u \dx }
		&\le \int_U \Big(C_L |\Grad u| + C_Q\big(|\Grad u|+1) \Big)\, 2\zeta \, |\Grad\zeta| \, |u| \dx\\
		&\le \frac{C_M}{4} \int_U \zeta^2 |\Grad u|^2 \dx
			+ C \int_U u^2 \dx + C.
	\end{align*}
    Furthermore, the map $\bigabs{a(\cdot,\zero)}$ is bounded on $U$ by a constant depending only on $U$ and $a$, since \eqref{reg:ass:a:qLip} and the boundedness of $U$ imply for all $\x\in U$ that
	\begin{align*}
		\bigabs{a(\x,\zero)}
		\le \bigabs{a(\x,\zero) - a(\zero,\zero)} + \bigabs{a(\zero,\zero)}
		\le C_Q|\x| + \bigabs{a(\zero,\zero)}
        \le C.
	\end{align*}
    Proceeding similarly to \eqref{reg:lhs:2.summand}, this inequality and the assumption \eqref{reg:ass:a:mon} allow us to control the second summand on the right-hand side of \eqref{EST:a:1} by
	\begin{align*}
		&\int_U a(\cdot,\Grad u) \cdot \zeta^2\Grad u \dx \\
		&\quad\ge \int_U \big(a(\cdot,\Grad u) - a(\cdot,\zero) \big) \cdot \zeta^2\Grad u \dx
			- \int_U \bigabs{a(\cdot,\zero)} \, \zeta^2\, |\Grad u| \dx \\
		&\quad\ge \int_U C_M \zeta^2|\Grad u|^2 \dx
			- \frac{C_M}{4} \int_U \zeta^2|\Grad u|^2 \dx
			- C \\
		&\quad= \frac{3C_M}{4} \int_U \zeta^2|\Grad u|^2 \dx - C.
	\end{align*}
    We further use the weak formulation \eqref{reg:QLE:WF} tested with $v=\zeta^2u$ to bound the left-hand side of \eqref{EST:a:1} by 
	\begin{align*}
        \biggabs{ \int_U a(\cdot,\Grad u) \cdot \Grad (\zeta^2u) \dx }
		= \biggabs{ \int_U f \zeta^2u \dx }
		\le C \int_U f^2 + u^2 \dx.
	\end{align*}
    Combining the above estimates for the terms in \eqref{EST:a:1}, we eventually obtain 
	\begin{align*}
		\int_W |\Grad u|^2 \dx
		\le \int_U \zeta^2 |\Grad u|^2 \dx
		\le C \int_U f^2 + u^2 \dx + C.
	\end{align*}
    Together with \eqref{reg:est:with:Gradu}, the estimate
	\begin{align*}
		\norm{u}_{H^2(V)}
		&\le C\big(\norm{f}_{L^2(W)} + \norm{\nabla u}_{L^2(W)} + \norm{u}_{L^2(W)} +1 \big) \\
		&\leq C\big(\norm{f}_{L^2(U)} +\norm{u}_{L^2(U)} + 1\big)
	\end{align*}
    is established. Since the set $V\ssubset U$ was arbitrary, this proves that $u\in H^2_\mathrm{loc}(U)$, and for every open subset $V\ssubset U$ the estimate \eqref{reg:EST:u:H2V} holds as claimed.
\end{proof}

\subsection{Regularity up to the boundary}

\begin{theorem} \label{regbdr:THRM:reg:bdr}
    In addition to the assumptions \eqref{reg:ass:dom}--\eqref{reg:ass:f}, we assume that $U$ is of class $C^{1,1}$. Let $u\in H^1(U)$ be a weak solution of \eqref{reg:QLE} in the sense of Definition~\ref{reg:DEF:WS}. Then, it holds
    \begin{align*}
        u\in H^2(U),
    \end{align*}
    and there exists a constant $C\ge0$, which depends only on $U$ and $a$, such that the estimate 
    \begin{align}
        \label{regbdr:EST:u:H2U}
        \norm{u}_{H^2(U)}
        \le C \big( \norm{f}_{L^2(U)} + \norm{u}_{L^2(U)} + 1 \big)
    \end{align}
    holds. 
\end{theorem}

To prove Theorem~\ref{regbdr:THRM:reg:bdr}, we use a modified version of Lemma~\ref{reg:LMM:DQ} for difference quotients near the boundary which reads as follows.

\begin{lemma} \label{regbdr:LMM:DQ}
    Let $U\subset\R^n$ be an open half-ball in the open upper half-space
    $\R^n_+ := \{ \x=(x_1,\dots,x_n)\in\R^n \,\vert\, x_n>0\}$ with $n\in\N$, i.e., $U := B_r(\zero)\cap\R^n_+$ for a radius $r>0$, and let $V := B_s(\zero)\cap\R^n_+$ be an open half-ball with a smaller radius $0<s<r$. Under these assumptions, it holds:
    \begin{enumerate}[label=\textnormal{(\alph*)}]
		\item \label{regbdr:DQ.a}
		If $w_1,w_2\in L^2(U)$ are functions with $\supp w_i\subset V$ for $i=1$ or $i=2$, for all $0<\abs{h}< \dist(V,\partial U)$ and all $k\in\otonmo$, the identity
		\begin{align*}
			\int_U w_1\Grad_k^{-h}w_2 \dx
			= -\int_U \Grad_k^{h}w_1 w_2 \dx
		\end{align*}
        holds.
		\item\label{regbdr:DQ.b}
		Suppose $w\in H^1(U)$ and $k\in\otonmo$. Then, for each half-ball $U':= B_{r'}(\zero)\cap\R^n_+$ with a radius $0<r'<r$, the difference quotient $\dkh w$ belongs to $L^2(U')$ for all $0<\abs{h}< \dist(U',\partial U)$, and it satisfies the estimate
		\begin{align*}
			\bignorm{\dkh w}_{L^2(U')}
			\leq \norm{\del_k w}_{L^2(U)}.
		\end{align*}
		\item \label{regbdr:DQ.c}
		Let $w\colon U\to\R^n$ be locally integrable with $w\in L^2(V)$ and let $k\in{\otonmo}$. Assume further that there exists a constant $C\ge 0$ such that $\dkh w \in L^2(V)$ and $\norm{\dkh w}_{L^2(V)} \le C$ for all $0<\abs h< \dist(V,\partial U)$. Then, the weak partial derivative $\del_k w$ exists and it holds
		\begin{align*}
			\del_k w\in L^2(V) \text{ with }
			\norm{\del_k w}_{L^2(V)} \le C.
		\end{align*}
	\end{enumerate}
    Moreover, if now $U\subset\R^n$ with $n\in\N$ is a bounded open subset of $\R^n$, the following holds.
    \begin{enumerate}[label=\textnormal{(\alph*)}]
        \setcounter{enumi}{3}
		\item \label{regbdr:DQ.d}
        Let $w\in L^2(U)$ and $k\in\oton$, and assume there exists a constant $C\ge 0$ such that $\dkh w \in L^2(U')$ and $\norm{\dkh w}_{L^2(U')} \le C$ for all open subsets $U'\ssubset U$ and all $0<\abs h< \dist(U',\partial U)$. Then, the weak partial derivative $\del_k w$ exists and it holds
		\begin{align*}
			\del_k w\in L^2(U) \text{ with }
			\norm{\del_k w}_{L^2(U)} \le C.
		\end{align*}
    \end{enumerate}
\end{lemma}

\begin{proof}
    The statement \ref{regbdr:DQ.a} can be established by a change of variables. According to a remark after the proof of \cite[§5.8.2 Theorem~3]{Evans2010}, the results \ref{regbdr:DQ.b} and \ref{regbdr:DQ.c} remain true in the given setting. Moreover, the result \ref{regbdr:DQ.d} is shown in \cite[Lemma~7.24]{Gilbarg2001}.
\end{proof}

\begin{proof}[Proof of Theorem~\ref{regbdr:THRM:reg:bdr}.]
    The proof is split into three steps.

    \textit{Step~1.}
    We first study the case of $U$ being a half-ball in the open upper half-space, i.e.,
	\begin{equation*}
		U=B_\varrho(\zero)\cap\R^n_+
	\end{equation*}
    with a radius $\varrho>0$. In this step, the letter $C$ will denote generic non-negative constants depending only on $a$ and $\varrho$. Furthermore, we fix the half-balls $V := B_{\varrho/4}(\zero)\cap\R^n_+$, $W' := B_{\varrho/2}(\zero)\cap\R^n_+$ and $W := B_{3\varrho/4}(\zero)\cap\R^n_+$. Next, we choose a cutoff function $\zeta\in C^\infty(\R^n)$ with $0\le\zeta\le1$ such that $\zeta=1$ in $\overline{B_{\varrho/2}(\zero)}$ and $\zeta=0$ in \revised{$\R^n\setminus \overline{B_{3\varrho/4}(\zero)}$}. Thus, $\zeta$ vanishes near the curved part of $\partial U$. Proceeding as in the proof of Theorem~\ref{reg:THRM:reg:int}, we test the weak formulation \eqref{reg:QLE:WF} with a second-order difference quotient $v=-\dkmh (\zeta^2 \dkh u)$, where $k\in\otonmo$ is arbitrary and $0<|h|<\frac{\varrho}{4}$. Analogously to that proof, invoking Lemma~\ref{regbdr:LMM:DQ}\ref{regbdr:DQ.a} and \ref{regbdr:DQ.b}, we obtain
	\begin{align}
	\label{regbdr:DQ:Gradu:k<n}
		\bignorm{\dkh\nabla u}_{L^2(V)}
		\le C\big( \norm{f}_{L^2(W)} + \norm{\nabla u}_{L^2(W)} +1\big).
	\end{align}
    Lemma~\ref{regbdr:LMM:DQ}\ref{regbdr:DQ.c} then implies $\partial_k (\nabla u) \in L^2(V)$ for all $k\in\otonmo$ along with the estimate
	\begin{align}
	\label{regbdr:PD:2nd:ord}
		\sum_{\substack{k,l=1 \\ k+l<2n}}^n \norm{\partial_k \partial_l u}_{L^2(V)}
		\le C\big( \norm{f}_{L^2(W)} + \norm{\nabla u}_{L^2(W)} +1\big).
	\end{align}
    In particular, we have $\partial_k u\in H^1(V)$ for all $k\in\otonmo$. Lemma~\ref{reg:LMM:DQ}\ref{reg:DQ.b} thus yields $\dkh \partial_k u\in L^2(V')$ for every open subset $V'\ssubset V$, and for all $0<|h|<\dist(V',\partial V)$, it holds
	\begin{align}
	\label{regbdr:dnh(del_ku)}
		\sum_{k=1}^{n-1} \bignorm{\dkh \partial_k u}_{L^2(V')}
		\le \sum_{k=1}^{n-1} \norm{\partial_n \partial_k u}_{L^2(V)}
		\le C\big( \norm{f}_{L^2(W)} + \norm{\nabla u}_{L^2(W)} +1\big).
	\end{align}

    Our next goal is to also control $\norm{\partial_n \partial_n u}_{L^2(V)}$. For this purpose, let $v\in C_c^\infty(V)$ be an arbitrary function, and fix $k\in \otonmo$ and $l\in\oton$. For the component $a_l$ of $a$, employing Lemma~\ref{regbdr:LMM:DQ}\ref{regbdr:DQ.a} as well as the assumptions \eqref{reg:ass:a:Lip} and \eqref{reg:ass:a:qLip} (which are also true for the component $a_l$), we deduce from the estimate \eqref{regbdr:DQ:Gradu:k<n} that
	\begin{align*}
		&\biggabs{ \int_U a_l\big(\x,\Grad u(\x)\big) \dkmh v \dx } 
		= \biggabs{ \int_V \dkh \big(a_l\big(\x,\Grad u(\x)\big) \big) v \dx } \nonumber\\
		&\le \int_V \bigg( 
            \frac{1}{|h|} \Bigabs{ a_l\big(\x+h\e_k, \Grad u(\x+h\e_k)\big) - a_l\big(\x+h\e_k, \Grad u(\x)\big) } \nonumber\\
		&\quad + \frac{1}{|h|} \Bigabs{ a_l\big(\x+h\e_k, \Grad u(\x)\big) - a_l\big(\x, \Grad u(\x)\big) }
		  \bigg) |v| \dx \nonumber\\
		&\le \int_V \bigg( \frac{1}{|h|} C_L \bigabs{ \Grad u(\x+h\e_k) - \Grad u(\x) }   + \frac{1}{|h|} C_Q \big(|\Grad u(\x)|+1\big) |h\e_k|
		  \bigg) |v| \dx \nonumber\\
		&= \int_V \Big(C_L \bigabs{\dkh \Grad u} + C_Q\big(|\Grad u|+1\big) \Big) |v| \dx \nonumber\\[1ex]
		&\le C\big( \bignorm{\dkh \Grad u}_{L^2(V)} + \norm{\Grad u}_{L^2(V)} +1\big)  \norm{v}_{L^2(V)}\\[1ex]
		&\le C\big( \norm{f}_{L^2(W)} + \norm{\Grad u}_{L^2(W)} +1\big) \norm{v}_{L^2(V)}.
	\end{align*}
    By means of the mean value theorem, the term $|\dkmh v|^2$ with $0<|h|< \dist(V,\partial U)$ can be bounded from above by an integrable majorant that is independent of $h$. Together with the pointwise convergence $\dkmh v \to \partial_k v$, Lebesgue's general convergence theorem (see, e.g., \cite[Section~3.25]{Alt}) leads to $\dkmh v \to \partial_k v$ in $L^2(V)$ as $h\to0$. Thus, by passing to the limit in the above inequality, we conclude
	\begin{align}
    \label{EST:BREG:1}
		\biggabs{ \int_V a_l(\cdot,\Grad u) \partial_k v \dx }
		\le C\big( \norm{f}_{L^2(W)} + \norm{\Grad u}_{L^2(W)} +1\big) \norm{v}_{L^2(V)}.
	\end{align}
    Using this estimate as well as the weak formulation \eqref{reg:QLE:WF}, we infer
	\begin{align}
    \label{EST:BREG:2}
		\biggabs{ \int_V a_n(\cdot,\Grad u) \partial_n v \dx }
		&= \biggabs{ \int_V \sum_{k=1}^{n-1} a_k(\cdot,\Grad u) \partial_k v - fv\dx }\nonumber\\[1ex]
		&\le C\big( \norm{f}_{L^2(W)} + \norm{\Grad u}_{L^2(W)} +1 \big) \norm{v}_{L^2(V)}.
	\end{align}

    We now fix an arbitrary $k\in\oton$, and we define the linear functional
	\begin{align*}
        T_k: C_c^\infty(V) \to\R,\quad
		T_k(v)
		:= \int_{V} a_n(\cdot,\Grad u) \partial_k v \dx.
	\end{align*}
    By the estimates \eqref{EST:BREG:1} and \eqref{EST:BREG:2}, we see that for all functions $v\in C_c^\infty(V)$, it holds
	\begin{align}
		\label{regbdr:T_k(v)}
		|T_k(v)|
		\le C\big( \norm{f}_{L^2(W)} + \norm{\Grad u}_{L^2(W)} +1\big) \norm{v}_{L^2(V)}.
	\end{align}
    In the following, since we intend to obtain this estimate also for all functions in $L^2(V)$, let $v\in L^2(V)$ be arbitrary. We proceed as in \cite[Theorem~4.15(3)]{Alt} to approximate $v$ as follows. For any function $\varphi\in C_c^\infty(\R^n)$ with $\supp \varphi \subset B_1(0)$, $\varphi\ge0$ and $\int_{\R^n} \varphi \dx =1$, we construct a standard Dirac sequence $(\varphi_i)_{i\in\N}$ by defining $\varphi_i(\x) := i^n \varphi(i\x)$ for all $\x\in\R^n$. For every $i\in\N$, we further introduce the set 
    $$V^{(i)} := \big\{\x\in V \,\vert\, \dist(\x,\partial V) > \tfrac{1}{i}\big\} \cap B_i(\mathbf{0}).$$ 
    It clearly holds $v_i:= \big(\mathds{1}_{V^{(i)}} v\big) \ast \varphi_i \in C_c^\infty(V)$ for all $i\in\N$ and
    according to \cite[Theorem~4.15]{Alt}, we also have
	\begin{align}
	\label{regbdr:approx:v}
		v_i \to v \text{ in } L^2(V) \quad\text{as}\quad i\to \infty.
	\end{align}
    Due to \eqref{regbdr:T_k(v)} and \eqref{regbdr:approx:v}, $(T_k(v_i))_{i\in\N}$ is a Cauchy sequence as it holds
	\begin{align*}
		&\bigabs{T_k(v_i) - T_k(v_j)}
		= \bigabs{T_k(v_i - v_j)}\\
		&\quad \le C\big( \norm{f}_{L^2(W)} + \norm{\Grad u}_{L^2(W)} +1\big) \norm{v_i - v_j}_{L^2(V)}
		\to 0
	\end{align*}
	as $i,j\to\infty$. Consequently, there exists a scalar $T_{k,v}\in\R$ such that $T_k(v_i)\to T_{k,v}$ as $i\to\infty$. This allows us to define the extended functional $\tilde{T}_k$ on $L^2(V)$ by
	\begin{align*}
		\tilde{T}_k(v) 
		:=
		\begin{cases}
			T_k(v) &\text{if } v\in C_c^\infty(V),\\
			T_{k,v} &\text{otherwise}.
		\end{cases}
	\end{align*}
    This functional is linear and continuous on $L^2(V)$ because for all $v\in L^2(V)$, we infer from the estimate \eqref{regbdr:T_k(v)} and the approximation of $v$ from \eqref{regbdr:approx:v} that 
	\begin{align*}
		|T_{k,v}|
		= \lim_{i\to\infty} |T_k(v_i)| 
		&\le C\big( \norm{f}_{L^2(W)} + \norm{\Grad u}_{L^2(W)} +1\big) \lim_{i\to\infty} \norm{v_i}_{L^2(V)} \\
		&= C\big( \norm{f}_{L^2(W)} + \norm{\Grad u}_{L^2(W)} +1\big) \norm{v}_{L^2(V)}.
	\end{align*}
    The Riesz representation theorem thus ensures the existence of a function $w\in L^2(V)$ such that for all $v\in L^2(V)$, the identity
	\begin{align}
		\label{regbdr:Riesz}
		\tilde{T}_k(v)
		= (w,v)_{L^2(V)}
	\end{align}
    is satisfied. 
    In particular, we deduce the estimate
	\begin{align*}
		\norm{w}_{L^2(V)}^2
		= \bigabs{\tilde{T}_k(w)}
		\le C\big( \norm{f}_{L^2(W)} + \norm{\Grad u}_{L^2(W)} +1\big) \norm{w}_{L^2(V)},
	\end{align*}
	which directly yields
	\begin{align}
		\label{regbdr:est:w}
		\norm{w}_{L^2(V)}
		\le C\big( \norm{f}_{L^2(W)} + \norm{\Grad u}_{L^2(W)} +1\big).
	\end{align}
	Recalling the definition of the functional $T_k$ as well as \eqref{regbdr:Riesz}, we infer the identity
	\begin{align*}
		\int_{V} a_n(\cdot,\Grad u) \partial_k v \dx
		= T_k(v)
		= \tilde{T}_k(v)
		= \int_{V} vw \dx
	\end{align*}
	for all $v\in C_c^\infty(V)$. This proves that $-w$ is the distributional derivative of $a_n(\cdot,\Grad u)$ with respect to $\x_k$ meaning that $\partial_k (a_n(\cdot,\Grad u))=-w \in L^2(V)$ for all $k\in\oton$. We have thus shown $a_n(\cdot,\Grad u)\in H^1(V)$. In combination with Lemma~\ref{reg:LMM:DQ}\ref{reg:DQ.b} and the estimate \eqref{regbdr:est:w}, we conclude that on every open subset $V'\ssubset V$, the $n$-th difference quotient of $a_n(\cdot,\Grad u)$ is uniformly bounded by
	\begin{align}
		\label{regbdr:est:dnh(a_n)}
		&\bignorm{ \dnh \big(a_n(\cdot,\Grad u)\big) }_{L^2(V')}
		\le \bignorm{ \partial_n \big(a_n(\cdot,\nabla u)\big) }_{L^2(V)} \nonumber\\
		&\quad = \norm{w}_{L^2(V)} 
		\le C\big( \norm{f}_{L^2(W)} + \norm{\nabla u}_{L^2(W)} +1\big)
	\end{align}
    for all $0<|h|<\dist(V',\partial V)$.

    In order to eventually control the desired difference quotient $\dnh \partial_n u$, we first use the strong monotonicity of $a$ with respect to $\p$ to derive the estimate
	\begin{align*}
		&C_M \bigabs{\partial_n u(\x+h\e_n) - \partial_n u(\x) }^2\\[1ex]
		&\le \bigg( a \Big(
			\x+h\e_n,
			\big(\partial_1 u(\x), \dots, \partial_{n-1} u(\x), \partial_n u(\x+h\e_n) \big)^T
		  \Big)
		  \\
		&\qquad - a \Big( \x+h\e_n, \Grad u(\x) \Big)
		\bigg) \cdot \big( 0,\dots,0, \partial_n u(\x+h\e_n) - \partial_n u(\x) \big)^T \\[1ex]
		&= \bigg( a_n \Big(
			\x+h\e_n,
			\big(\partial_1 u(\x), \dots, \partial_{n-1} u(\x), \partial_n u(\x+h\e_n) \big)^T
		  \Big)
		   \\
		&\qquad - a_n \Big( \x+h\e_n, \Grad u(\x) \Big)
		\bigg) \big( \partial_n u(\x+h\e_n) - \partial_n u(\x) \big) \\[1ex]
		&\le \bigg| a_n \Big(
			\x+h\e_n,
			\big(\partial_1 u(\x), \dots, \partial_{n-1} u(\x), \partial_n u(\x+h\e_n) \big)^T
		  \Big)
		  \\
		&\qquad - a_n \Big( \x+h\e_n, \Grad u(\x) \Big)
		\bigg| \; \bigabs{ \partial_n u(\x+h\e_n) - \partial_n u(\x) }
	\end{align*}
	for all $\x\in V$. Using the Lipschitz continuity of $a_n$ with respect to $\p$ and the quasi Lipschitz continuity of $a_n$ with respect to $\x$, we infer
	\begin{align*}
		&C_M\bigabs{\dnh \partial_n u(\x)}
        \\[1ex]
		&= \frac{C_M}{|h|} \bigabs{\partial_n u(\x+h\e_n) - \partial_n u(\x) } \\
		&\le \frac{1}{|h|}
            \bigg| a_n \Big(
				\x+h\e_n,
				\big(\partial_1 u(\x), \dots, \partial_{n-1} u(\x), \partial_n u(\x+h\e_n) \big)^T
			\Big) 
        \\
        &\qquad\qquad 
            - a_n \Big( \x+h\e_n, \Grad u(\x) \Big)
		\bigg| 
        \\[1ex]
		&\le \frac{1}{|h|}
			\bigg| a_n \Big(
				\x+h\e_n,
				\big(\partial_1 u(\x), \dots, \partial_{n-1} u(\x), \partial_n u(\x+h\e_n) \big)^T
			\Big) 
        \\
        &\qquad\qquad
		    - a_n \Big( \x+h\e_n, \Grad u(\x+h\e_n) \Big)
		\bigg| \\
		&\quad+  \frac{1}{|h|} 
			\Bigabs{ a_n \big( \x+h\e_n, \Grad u(\x+ h\e_n) \big)
			- a_n \big( \x+h\e_n, \Grad u(\x) \big)
		} 
        \\[1ex]
		&\le C_L \sum_{k=1}^{n-1} \frac{1}{|h|} \bigabs{\partial_k u(\x+h\e_n) 
			- \partial_k u(\x) } \\
		&\quad+ \frac{1}{|h|} \Bigabs{a_n \big(\x +h\e_n, \Grad u(\x+h\e_n) \big) 
			- a_n \big(\x, \Grad u(\x) \big) } \\
		&\quad+ \frac{1}{|h|} \Bigabs{a_n \big(\x +h\e_n, \Grad u(\x) \big) 
			- a_n \big(\x, \Grad u(\x) \big) } 
        \\[1ex]
		&= C_L \sum_{k=1}^{n-1} \bigabs{\dnh \partial_k u(\x) } + \bigabs{\dnh \big(a_n\big(\x,\Grad u(\x)\big) \big)} + C_Q\big(|\Grad u(\x)|+1\big)
	\end{align*} 
	for all $\x\in V$. After taking the $L^2(V')$ norm on both sides for any open subset $V'\ssubset V$, we use the estimates \eqref{regbdr:dnh(del_ku)} and \eqref{regbdr:est:dnh(a_n)} to gain control of the difference quotient by
	\begin{align*}
		&\bignorm{\dnh \partial_n u}_{L^2(V')} \\
		&\quad\le C \Bigg (\sum_{k=1}^{n-1} \bignorm{\dnh \partial_k u}_{L^2(V')}  + \bignorm{\dnh \big(a_n(\cdot,\Grad u)\big)}_{L^2(V')} + \norm{\Grad u}_{L^2(V')}+1 \bigg)\\
		&\quad\le C\big( \norm{f}_{L^2(W)} + \norm{\Grad u}_{L^2(W)} +1\big)
	\end{align*}
    for every $V'\ssubset V$ and all $0<|h|<\dist(V',\partial V)$. Eventually, according to Lemma~\ref{regbdr:LMM:DQ}\ref{regbdr:DQ.d}, we obtain $\partial_n \partial_n u \in L^2(V)$ as well as the estimate
	\begin{align*}
		\norm{\partial_n \partial_n u}_{L^2(V)}
		\le C\big( \norm{f}_{L^2(W)} + \norm{\Grad u}_{L^2(W)} +1\big).
	\end{align*}
    Together with the estimate \eqref{regbdr:PD:2nd:ord} for the remaining second-order partial derivatives, we finally conclude
	\begin{align*}
		\norm{u}_{H^2(V)}^2
		&= \sum_{k,l=1}^n \norm{\partial_k \partial_l u}_{L^2(V)}^2 + \norm{\Grad u}_{L^2(V)}^2 + \norm{u}_{L^2(V)}^2\\
		&\le C\big( \norm{f}_{L^2(W)}^2 + \norm{\Grad u}_{L^2(W)} ^2 + \norm{ u}_{L^2(W)}^2  +1\big).
	\end{align*}
    As in Step~2 of the proof of Theorem~\ref{reg:THRM:reg:int}, we further estimate the term $\norm{\Grad u}_{L^2(W)}$. Since no results for the $n$-th difference quotient were used for proving this estimate in the interior, we can proceed analogously. On the half-ball $V$, we hence have $u\in H^2(V)$ with
	\begin{align*}
		\norm{u}_{H^2(V)}
		\le C\big(\norm{f}_{L^2(U)} +\norm{u}_{L^2(U)} + 1\big),
	\end{align*}
    where the constant $C\ge0$ only depends on $a$ and $\varrho$.

    \textit{Step~2.} 
    Instead of a half-ball, we now consider a general bounded open subset $U\subset\R^n$ of class $C^{1,1}$ as given in the assumptions. Near an arbitrary point $\x^0\in\partial U$ we straighten the boundary by means of a $C^{1,1}$-diffeomorphism
	\begin{align*}
		\tau: U\cap B_r \big(\x^0\big) \to \tau\Big(U\cap B_r\big(\x^0\big)\Big) \subset \R^n_+
	\end{align*}
    with a radius $r>0$. The corresponding point in the image is called $\y^0:=\tau(\x^0)$. For every $\x\in \partial U\cap B_r\big(\x^0\big)$, it follows for the $n$-th component of $\y:=\tau(\x)$ that $\y_n=0$. Without loss of generality, we choose $\tau$ and $r$ such that the Jacobian determinant fulfills $\det D\tau(\x) =1$ for all $\x\in U\cap B_r\big(\x^0\big)$, where $\tau$ and $r$ depend exclusively on $U$ and $\x^0$. Next, we select a radius $s>0$ so small that the half-ball $\hat{U}:= B_s\big(\y^0\big) \cap \R^n_+$ is contained in the image $\tau\big(U\cap B_r\big(\x^0\big)\big)$, and we set $\hat{V}:= B_{s/4}\big(\y^0\big) \cap \R^n_+$. Consequently, $s$ depends only on $U$ and $\x^0$.

    For all $\y\in \hat{U}$ and $\p\in\R^n$, we denominate the transformed functions by
	\begin{align*}
		\hat{u}(\y) &:= u\big(\tau^{-1}(\y)\big), \\
		\hat{f}(\y) &:= f\big(\tau^{-1}(\y)\big), \\
		\hat{a}(\y,\p) &:= D\tau\big(\tau^{-1}(\y)\big) \, a\Big(\tau^{-1}(\y), (D\tau)^T\big(\tau^{-1}(\y)\big) \p \Big).
	\end{align*}
    The form of $\hat{a}$ is deduced from the change of variable formula as follows: Let $\hat{v}\in H^1(\hat{U})$ be arbitrary. Then, $v(\x) := \hat{v}\big(\tau(\x)\big)$ is an admissible test function in $H^1\big(\tau^{-1}(\hat{U})\big)$ and due to $\det D\tau\big(\tau^{-1}(\y)\big)=1$, the identity
	\begin{align*}
		&\int_{\tau^{-1}(\hat{U})} a\big(\x, \Grad u(\x) \big) \cdot \Grad v(\x) \dx \\
		&= \int_{\tau^{-1}(\hat{U})} a\big(\x, \Grad (\hat{u}\circ\tau)(\x) \big) \cdot \Grad (\hat{v}\circ\tau)(\x) \dx \\
		&= \int_{\tau^{-1}(\hat{U})} a\Big(\x, (D\tau)^T(\x) \Grad \hat{u}\big(\tau(\x)\big) \Big) \cdot (D\tau)^T(\x) \Grad \hat{v} \big(\tau(\x)\big) \dx \\
		&= \int_{\hat{U}} a\Big(\tau^{-1}(\y), (D\tau)^T\big(\tau^{-1}(\y)\big) \Grad \hat{u}(\y) \Big) \cdot (D\tau)^T\big(\tau^{-1}(\y)\big) \Grad \hat{v}(\y) \dy \\
        &= \int_{\hat{U}} D\tau\big(\tau^{-1}(\y)\big) \, a\Big(\tau^{-1}(\y), (D\tau)^T\big(\tau^{-1}(\y)\big) \Grad \hat{u}(\y) \Big) \cdot \Grad \hat{v}(\y) \dy
	\end{align*}
    holds. This means that for all $\hat{v}\in H^1(\hat{U})$, the equation
	\begin{align*}
		\int_{\hat{U}} \hat{a}(\cdot,\Grad \hat{u}) \cdot \Grad \hat{v} \dy
		= \int_{\hat{U}} \hat{f}\hat{v} \dy
	\end{align*}
    is fulfilled. A straightforward computation invoking the properties of the $C^{1,1}$-diffeomorphism $\tau$ shows that $\hat{a}$ satisfies the assumptions \eqref{reg:ass:a:Lip}--\eqref{reg:ass:a:qLip} as well. For the thereby occurring constants, we have a dependency on $U$, $a$ and $\x$ since $\tau$ depends only on $U$ and $\x$.  

    Thus, $\hat{u}$ is a weak solution of the boundary value problem
    \begin{alignat*}{2}
        -\Grad \cdot \hat{a}(\cdot,\Grad \hat{u}) &= \hat{f}
        &&\quad\text{in $\hat{U}$},\\
        \hat{a}(\cdot,\Grad \hat{u}) \cdot \n &= 0
        &&\quad\text{on $\partial \hat{U}$},
    \end{alignat*}
    where the assumptions \eqref{reg:ass:dom}--\eqref{reg:ass:f} are satisfied by $\hat{U}$, $\hat{a}$ and $\hat{f}$ and $\hat{U}$ is a half-ball in the open upper half-space. Therefore, Step~1 implies the regularity $\hat{u}\in H^2(\hat{V})$ along with the estimate
	\begin{align}
		\label{regbdr:est:hat(u):H^2(hat(V))}
		\norm{\hat{u}}_{H^2(\hat{V})}
		\le \hat{C}\big(\norm{\hat{f}}_{L^2(\hat{U})} +\norm{\hat{u}}_{L^2(\hat{U})} + 1\big)
	\end{align}
    with a constant $\hat{C}\ge0$ depending only on $\hat{a}$ and $s$. As $\hat{a}$ depends only on $a$ and $\tau$ and thus only on $U$, $a$ and $\x^0$, and since $s$ depends only on $U$ and $\x^0$, we conclude that $\hat{C}$ actually depends only on $U$, $a$ and $\x^0$. 

    We now retransform the set $\hat{V}$ and consider its preimage
	\begin{align*}
		V:=\tau^{-1}(\hat{V}),
	\end{align*}
    which depends only on $U$ and $\x^0$ and contains the point $\x^0$. The $C^{1,1}$-diffeo\-morphism $\tau$ passes the property of $\hat{V}$ having a Lipschitz boundary on to $V$. Since in particular, $\tau\in C^{1,1}(V,\hat{V})$ is a $C^1$-diffeomorphism, the chain rule for Sobolev functions (see, e.g. \cite[Theorem~4.26]{Alt}) yields on the one hand $\Grad\hat{u}\circ\tau \in H^1(V;\R^n)$, and on the other hand, for the composition $u=\hat{u}\circ\tau$, the identity
	\begin{align*}
		\Grad u
		= \Grad(\hat{u}\circ\tau)
		= (D\tau)^T \Grad\hat{u}\circ\tau.
    \end{align*}
    According to \cite[Theorem~10.5]{Alt}, $(D\tau)^T$ belongs to $W^{1,\infty}(V;\R^n)$. Eventually, we conclude by the product rule (see, e.g. \cite[Theorem~4.25]{Alt}) that $\Grad u\in H^1(V;\R^n)$ which shows $u\in H^2(V)$. From the chain rule and $\tau\in W^{2,\infty}(V;\R^n)$ we moreover deduce the existence of a constant $C_\tau\ge0$ depending only on $\tau$ and thus on $U$ und $\x^0$ such that $\norm{u}_{H^2(V)} \leq C_\tau \norm{\hat{u}}_{H^2(\hat{V})}$ holds. The transformation formula and the properties of $\tau$ further imply the estimate $$\norm{\hat{f}}_{L^2(\hat{U})} = \norm{f}_{L^2(\tau^{-1}(\hat{U}))} \leq \norm{f}_{L^2(U)}.$$ 
    An analogous estimate holds for $\hat{u}$. In summary, \eqref{regbdr:est:hat(u):H^2(hat(V))} leads to the bound
	\begin{align}
		\label{regbdr:est:u:H^2(hat(V))}
		\norm{u}_{H^2(V)}
		&\le C_\tau \norm{\hat{u}}_{H^2(\hat{V})} \nonumber\\
		&\leq  C_\tau \hat{C} \big(\norm{\hat{f}}_{L^2(\hat{U})} +\norm{\hat{u}}_{L^2(\hat{U})} + 1\big) \nonumber\\
		&\leq C_{\x^0}\big(\norm{f}_{L^2(U)} +\norm{u}_{L^2(U)} + 1\big),
	\end{align}
    where the constant $C_{\x^0}\ge 0$ depends only on $U$, $a$ and $\x^0$. 

    \textit{Step~3.}
    According to Step~2, for every point $\x\in \partial U$, there is an open set $V_{\x}$ containing $\x$ such that we have the regularity $u\in H^2(V_{\x})$ along with the estimate \eqref{regbdr:est:u:H^2(hat(V))} written for $\x$. Consequently, the boundary of $U$ has the cover $\partial U \subset \bigcup_{\x\in \partial U} V_{\x}$. Since $\partial U$ is compact, we are able to choose finitely many points $\x^1,\dots,\x^m\in\partial U$ such that the respective sets $V_{\x^1},\dots,V_{\x^m}$ cover $\partial U$. Moreover, there exists an open subset $V_\mathrm{Int}\ssubset U$ such that the entire set $U$ is covered by
	\begin{align*}
		U\subset V_\mathrm{Int} \cup \bigcup_{j=0}^m V_{\x^j}.
	\end{align*}
    We point out that the points $\x^1,\dots,\x^m$ as well as the set $V_\mathrm{Int}$ depend only on $U$.

    We now merge the previous results as follows. On the one hand, we use the estimate \eqref{reg:EST:u:H2V} corresponding to the regularity in the interior from Theorem~\ref{reg:THRM:reg:int} for the set $V_\mathrm{Int}$. Here, the associated constant $C_{V_\mathrm{Int}}\ge0$ depends on $U$, $V_\mathrm{Int}$ and $a$, and thus only on $U$ and $a$. On the other hand, we apply the estimate \eqref{regbdr:est:u:H^2(hat(V))} resulting from Step~2 to the sets $V_{\x^1},\dots,V_{\x^m}$. Here, the respective constants $C_{\x^1},\dots,C_{\x^m}\geq0$ depend on $U$, $a$ and $\x^1,\dots,\x^m$. In summary, we obtain the bound
	\begin{align*}
		\norm{u}_{H^2(U)}
		&\le \norm{u}_{H^2(V_{In})} + \sum_{j=1}^m \norm{u}_{H^2(V_{\x^j})}\\
		&\le \bigg(C_{V_{In}} + \sum_{j=1}^m C_{\x^j} \bigg) \big(\norm{f}_{L^2(U)} +\norm{u}_{L^2(U)} + 1\big)\\
		&\le C \big(\norm{f}_{L^2(U)} +\norm{u}_{L^2(U)} + 1\big)
	\end{align*}
    with a constant $C\ge0$ depending only on $U$ and $a$. This proves the regularity $u\in H^2(U)$ along with the estimate \eqref{regbdr:EST:u:H2U} which completes the proof.
\end{proof}


\section{Regularity theory and separation properties for weak solutions of the anisotropic Cahn--Hilliard equation}
\label{SECT:REG:AICH}

In this final section, we present the proofs of Theorem~\ref{THM:REG} and Theorem~\ref{THM:REG:2}.

\begin{proof}[Proof of Theorem~\ref{THM:REG}]
    Let $\varphi$ and $\mu$ be arbitrary representatives of their equivalence class. From the weak formulation \eqref{DEF:WS:WF2} or \eqref{DEF:WS:WF2*}, respectively, it follows that
    \begin{align*}
        \intO A'\big(\Grad\varphi(t)\big) \cdot \Grad\eta \dx
        = \intO \big[ \mu(t) - F'\big(\varphi(t)\big) ] \,\eta \dx
    \end{align*}
    holds for all $\eta\in H^1(\Omega)$ and almost all $t\in [0,T]$. Recalling the regularity of a weak solution, we deduce that there exists a null set $\mathcal N\subset [0,T]$, such that for all $t\in [0,T]\setminus \mathcal N$, it holds
    \begin{align*}
        \mu(t) - F'\big(\varphi(t)\big) \in L^2(\Omega)
    \end{align*}
    and $\varphi(t)$ is a weak solution of the quasilinear elliptic equation
    \begin{alignat*}{2}
        - \Grad\cdot A'\big(\Grad\varphi(t)\big)
        &=  \mu(t) - F'\big(\varphi(t)\big) 
        &&\quad\text{in $\Omega$},
        \\
        A'\big(\Grad\varphi(t)\big) \cdot \n &= 0
        &&\quad\text{on $\Gamma$}.
    \end{alignat*}
    As the domain $\Omega$ and the function $A'$ have the required properties, we apply Theorem~\ref{regbdr:THRM:reg:bdr} with $a=A'$ to conclude that $\varphi(t) \in H^2(\Omega)$ with
    \begin{align}
    \label{EST:H2:PTW}
        \norm{\varphi(t)}_{H^2(\Omega)}
        \le C \big( \norm{\mu(t)}_{L^2(\Omega)} 
        + \norm{F'\big(\varphi(t)\big)}_{L^2(\Omega)} 
        + \norm{\varphi(t)}_{L^2(\Omega)} +1 \big)
    \end{align}
    for all $t\in [0,T]\setminus \mathcal N$. Integrating this inequality with respect to time from $0$ to $T$, we eventually obtain
    \begin{align*}
        \norm{\varphi}_{L^2(0,T;H^2(\Omega))}
        &\le C \big( \norm{\mu}_{L^2(0,T;L^2(\Omega))} 
        + \norm{F'(\varphi)}_{L^2(0,T;L^2(\Omega))}
        \notag\\
        &\qquad\quad 
        +\norm{\varphi}_{L^2(0,T;L^2(\Omega))} +1 \big).
    \end{align*}
    Due to the regularity of weak solutions, the right-hand side of this estimate is finite and thus, $\varphi\in L^2(0,T;H^2(\Omega))$. 
    
    We further recall that $A':\R^d\to\R^d$ is Lipschitz continuous and it holds $A'(\zero) = \zero$ since $A'$ is positively one-homogeneous.
    Since $\Grad\varphi(t) \in H^1(\Omega)$ for all $t\in [0,T]\setminus \mathcal N$, we use a differentiability result for the composition of a Lipschitz function with a Sobolev function (see \cite[Corollary~3.2]{Ambrosio-DalMaso}) to conclude $A'(\Grad\varphi(t)) \in H^1(\Omega)$ for all $t\in [0,T]\setminus \mathcal N$. This proves \eqref{REG:PHIA}.
    
    Moreover, after testing the weak formulation \eqref{DEF:WS:WF2} or \eqref{DEF:WS:WF2*} with an arbitrary test function $\eta \in C^\infty_c(\Omega)$ and integrating by parts, we obtain
    \begin{align*}
            \intO \mu\, \eta \dx 
            = \intO - \Grad\cdot A'(\Grad\varphi)\, \eta + F'(\varphi)\,\eta \dx
    \end{align*}
    a.e.~in $[0,T]$. By the fundamental lemma of the calculus of variations, we eventually conclude \eqref{EQ:MU:PTW}.  
    Thus, the proof is complete.
\end{proof}

\begin{proof}[Proof of Theorem~\ref{THM:REG:2}]
    As the mobility function $M$ is assumed to be constant, we simply set $M\equiv 1$ without loss of generality. In the following, the letter $C$ will denote generic positive constants depending only on $\varphi_0$, $\mu_0$, $F$ and the quantities introduced in \eqref{ass:dom}--\eqref{ass:M} which may change their value from step to step.

    \textit{Step~1.} We first show the regularities $\varphi\in H^1\big(0,T;H^1(\Omega)\big) \cap L^2(0,T;H^2(\Omega))$ and $\Grad\mu\in L^\infty(0,T;L^2(\Omega))$.

    For any function $f:(-\infty,T] \to X$ (where $X$ is a Banach space), $t\in (-\infty,T]$ and $h>0$, we write
    \begin{equation*}
        \delth f(t) := \frac{1}{h}\big[ f(t) - f(t-h) \big] \in X
    \end{equation*}
    to denote the backward difference quotient of $f$ at time $t$.
    We now consider $\varphi$ and $\mu$ as arbitrary but fixed representatives of their respective equivalence class which can be evaluated at every time in $[0,T]$.
    For negative times, we extend $\varphi$ and $\mu$ by defining
    \begin{align*}
        \varphi(t):=\varphi_0,
        \quad
        \mu(t):=\mu_0,
        \quad
        \text{for all $t\le 0$}.
    \end{align*}
    Let now $0<h<T$ and $t\in [0,T]$ be arbitrary. We obtain
    \begin{align}
    \label{EST:REG:1}
        &\frac 12 \intO \left| \frac 1h \int_{t-h}^t \Grad \mu(\tau) \dtau \right|^2 \mathrm{d}\mathbf{x}
        \;-\; \frac 12 \intO \left| \Grad \mu_0 \right|^2 \dx
        \notag\\[1ex]
        &\quad = 
        \int_0^t \frac 12 \;\dds \intO \left| \frac 1h \int_{s-h}^s \Grad \mu(\tau) \dtau \right|^2 \dx \ds
        \notag\\[1ex]
        &\quad = 
        \int_0^t \intO \left(\frac 1h \int_{s-h}^s \Grad \mu(\tau) \dtau\right) 
        \cdot \dds \left(\frac 1h \int_{s-h}^s \Grad \mu(\tau) \dtau\right) \dx \ds 
        \notag\\[1ex]
        &\quad = 
        \int_0^t \intO \left(\frac 1h \int_{s-h}^s \Grad \mu(\tau) \dtau\right) 
        \cdot \delth \Grad\mu(s) \dx \ds 
        \notag\\[1ex]
        &\quad = 
        \int_0^t \mathds{1}_{[0,h]}(s) \; 
        \intO \left(\frac 1h \int_{s-h}^s \Grad \mu(\tau) \dtau\right) 
        \cdot \delth \Grad\mu(s) \dx \ds 
        \notag\\
        &\qquad + 
        \int_0^t \mathds{1}_{(h,T]}(s) \;
        \intO \left(\frac 1h \int_{s-h}^s \Grad \mu(\tau) \dtau\right) 
        \cdot \delth \Grad\mu(s) \dx \ds
        \notag\\[1ex]
        &\quad =: 
        I_1 + I_2.
    \end{align}
    Using the weak formulation \eqref{DEF:WS:WF1} or \eqref{DEF:WS:WF1*}, we deduce
    \begin{align}
    \label{EST:I1}
        I_1
        &=
        \int_0^t \mathds{1}_{[0,h]}(s) \! 
        \intO \left(
        \frac 1h \int_{0}^s \Grad \mu(\tau) \dtau
        + \frac 1h \int_{s-h}^0 \Grad \mu(\tau) \dtau
        \right) 
        \cdot \delth \Grad\mu(s) \dx \ds
        \notag\\[1ex]
        &=
        - \int_0^t \mathds{1}_{[0,h]}(s) \; 
         \left< \frac 1h \int_0^s\delt\varphi(\tau) \dtau , \delth\mu(s) \right> \ds
        \notag\\
        &\qquad 
        + \int_0^t \mathds{1}_{[0,h]}(s) \; 
        \intO \left(\frac 1h \int_{s-h}^0 \Grad \mu_0 \dtau\right) 
        \cdot \delth \Grad\mu(s) \dx \ds
        \notag\\[1ex]
        &=
        - \int_0^t \mathds{1}_{[0,h]}(s) \; \langle \delth\varphi(s), \delth\mu(s) \rangle \ds
        \notag\\
        &\qquad 
        + \int_0^t \mathds{1}_{[0,h]}(s) \; 
        \frac {h-s}h  \intO \Grad \mu_0 \cdot \delth \Grad\mu(s) \dx \ds.
    \end{align} 
    A similar computation reveals
    \begin{align}
    \label{EST:I2}
        I_2 
        = - \int_0^t \mathds{1}_{(h,T]}(s) \; \langle \delth\varphi(s), \delth\mu(s) \rangle \ds.
    \end{align}
    In summary, we thus have
    \begin{align}
    \label{EST:REG:2}
        &\frac 12 \intO \left| \frac 1h \int_{t-h}^t \Grad \mu(\tau) \dtau \right|^2 \dx
        \notag\\[1ex]
        &\quad =
        \frac 12 \norm{\Grad\mu_0}_{L^2(\Omega)}^2
        - \int_0^t \langle \delth\varphi(s), \delth\mu(s) \rangle \ds
        \notag\\
        &\qquad 
        + \int_0^t \mathds{1}_{[0,h]}(s) \; 
        \frac {h-s}h  \intO \Grad \mu_0 \cdot \delth \Grad\mu(s) \dx \ds.
    \end{align}
    Invoking the weak formulation \eqref{DEF:WS:WF2} or \eqref{DEF:WS:WF2*} for $\mu(s)$ with $s\in(0,T)$ as well as condition~\eqref{ASS:MU} for $\mu(s)=\mu_0$ with $s\in (-T,0]$, we deduce that
    \begin{align}
    \label{EQ:REG:MU}
        \intO \mu(s) \eta(s) \dx = \intO A'\big(\Grad\varphi(s)\big) \cdot \Grad\eta(s) + F'\big(\varphi(s)\big) \eta(s) \dx
    \end{align} 
    for all $\eta\in L^2(-T,T;H^1(\Omega))$ and almost all $s\in[-T,T]$.
    Recalling the definition of the backward difference quotient, we use \eqref{EQ:REG:MU} to rewrite the second term on the right-hand of \eqref{EST:REG:2} side as
    \begin{align}
    \label{EQ:DUAL}
    &- \int_0^t \langle \delth\varphi(s), \delth\mu(s) \rangle \ds 
    = - \int_0^t \intO \delth\mu(s) \delth\varphi(s) \dx\ds 
    \notag\\[1ex]    
    &\quad= 
    - \int_0^t \intO \delth A'\big(\Grad\varphi(s)\big) \cdot \delth\Grad\varphi(s) 
    + \delth\big[ F'\big(\varphi(t)\big) \big] \delth \varphi(s) \dx\ds.
    \end{align}
    Due to \eqref{ass:F3} in the case of a regular potential or \eqref{DEC:LOG} in the case of the logarithmic potential, $F$ can be decomposed as the sum of functions $F_1$ and $F_2$, where $F_1'$ is monotonically increasing and $F_2'$ is Lipschitz continuous. Recalling that $A':\R^d\to\R^d$ is strongly monotone (see \eqref{ass:A}), we deduce the estimate 
    \begin{align}
    \label{EST:DUAL}
    &- \int_0^t \langle \delth\varphi(s), \delth\mu(s) \rangle \ds 
    \notag\\[1ex]
    &\quad\le  
    - a_0 \int_0^t \norm{\delth\Grad\varphi(s)}_{L^2(\Omega)}^2 \ds
    + L \int_0^t \norm{\delth\varphi(s)}_{L^2(\Omega)}^2 \ds,
    \end{align}
    where $L$ denotes the minimal Lipschitz constant for $F_2'$. Using the weak formulation as well as Young's inequality, we obtain 
    \begin{align}
    \label{EST:SUM2}
    & \int_0^t \norm{\delth\varphi(s)}_{L^2(\Omega)}^2 \ds
    = \int_0^t \left< \frac 1h \int_{\max\{0,s-h\}}^s \delt \varphi(\tau) \dtau , \delth\varphi(s) \right> \ds
    \notag\\[1ex]
    &\quad=
    \int_0^t \intO \left( \frac 1h \int_{s-h}^s \Grad\mu(\tau) \dtau - \frac 1h \int_{s-h}^{\max\{0,s-h\}} \Grad\mu(\tau) \dtau \right) \cdot \delth\Grad\varphi(s) \dx 
    \ds
    \notag\\[1ex]
    &\quad\le
    \int_0^t \intO \left( 
        \left| \frac 1h \int_{s-h}^s \Grad\mu(\tau) \dtau \right|
        + \big| \Grad\mu_0 \big|
        \right) 
    \big| \delth\Grad\varphi(s) \big| \dx \ds
    \notag\\[1ex]
    &\quad\le
    C \left[ \int_0^t \intO \left| \frac 1h \int_{s-h}^s \Grad\mu(\tau) \dtau \right|^2 \dx \ds
    + \int_0^t \norm{\Grad\mu_0}_{L^2(\Omega)}^2 \ds \right]
    \notag\\
    &\qquad+ \frac{a_0}{2L} \int_0^t \norm{\delth\Grad\varphi(s)}_{L^2(\Omega)}^2 \ds.
    \end{align}
    In view of \eqref{EST:DUAL}, we thus infer
    \begin{align}
    \label{EST:DUAL:2}
    - \int_0^t \langle \delth\varphi(s), \delth\mu(s) \rangle \ds 
    &\le- \frac{a_0}{2} \int_0^t \norm{\delth\Grad\varphi(s)}_{L^2(\Omega)}^2 \ds
    + CLT \norm{\Grad\mu_0}_{L^2(\Omega)}^2
    \notag\\[1ex]
    &\qquad
    + CL \int_0^t \intO \left| \frac 1h \int_{s-h}^s \Grad\mu(\tau) \dtau \right|^2 \dx \ds.
    \end{align}
    Now, combining \eqref{EST:REG:2} and \eqref{EST:DUAL:2}, we conclude that for all $t\in[0,T]$,
    \begin{align}
    \label{EST:REG:3}
        &\frac 12 \intO \left| \frac 1h \int_{t-h}^t \Grad \mu(\tau) \dtau \right|^2 \dx
        + \frac{a_0}{2} \int_0^t \norm{\delth\Grad\varphi(s)}_{L^2(\Omega)}^2 \ds
        \notag\\[1ex]
        &\quad \le 
        C \norm{\Grad\mu_0}_{L^2(\Omega)}^2
        + 2CL \int_0^t \frac 12 \intO \left| \frac 1h \int_{s-h}^s \Grad\mu(\tau) \dtau \right|^2 \dx \ds
        \notag\\
        &\qquad 
        + \int_0^t \mathds{1}_{[0,h]}(s) \; 
        \frac {h-s}h  \intO \Grad \mu_0 \cdot \delth \Grad\mu(s) \dx \ds.
    \end{align}

    We now choose an arbitrary time $t\in [0,h]$.
    Recalling that $\mu(s)=\mu_0$ if $s\le 0$, the third summand on the right-hand side of \eqref{EST:REG:3} can be bounded by
    \begin{align}
    \label{EST:REG:4}
        &\int_0^t \mathds{1}_{[0,h]}(s) \; 
        \frac {h-s}h  \intO \Grad \mu_0 \cdot \delth \Grad\mu(s) \dx \ds
        \notag\\[1ex]
        &\quad = 
        \intO \Grad\mu_0 \cdot \int_0^t \frac {h-s}{h^2} \mathds{1}_{[0,h]}(s) 
            \big[ \Grad\mu(s) - \Grad\mu_0 \big] \ds \dx 
        \notag\\[1ex]
        &\quad \le 
        \intO \Grad\mu_0 \cdot \int_{t-h}^t \frac {h-s}{h^2} \mathds{1}_{[0,h]}(s) 
            \Grad\mu(s) \ds \dx 
        \notag\\[1ex]
        &\quad \le 
        C \norm{\Grad\mu_0}_{L^2(\Omega)}^2
        + \frac 14 \intO \left| \frac 1h \int_{t-h}^t \Grad \mu(\tau) \dtau \right|^2 \dx.
    \end{align}
    Combining this estimate with \eqref{EST:REG:3}, we infer that
    \begin{align}
    \label{EST:REG:3*}
        &\frac 14 \intO \left| \frac 1h \int_{t-h}^t \Grad \mu(\tau) \dtau \right|^2 \dx
        \notag\\[1ex]
        &\quad \le 
        C \norm{\Grad\mu_0}_{L^2(\Omega)}^2
        + 4CL \int_0^t \frac 14 \intO \left| \frac 1h \int_{s-h}^s \Grad\mu(\tau) \dtau \right|^2 \dx \ds
    \end{align}
    holds for all $t\in [0,h]$. Hence, Gronwall's lemma implies
    \begin{align}
    \label{EST:REG:5}
        \frac 14 \intO \left| \frac 1h \int_{t-h}^t \Grad \mu(\tau) \dtau \right|^2 \dx
        \le C \norm{\Grad\mu_0}_{L^2(\Omega)}^2 \; \mathrm e^{4CLT} \le C
    \end{align}
    for all $t\in [0,h]$.   
    Consequently, due to \eqref{EST:REG:4}, we have
    \begin{align}
    \label{EST:REG:6}
        \int_0^t \mathds{1}_{[0,h]}(s) \; 
        \frac {h-s}h  \intO \Grad \mu_0 \cdot \delth \Grad\mu(s) \dx \ds \le C
    \end{align}
    for all $t\in [0,T]$. Using this estimate to bound the third summand on the right-hand side of \eqref{EST:REG:3}, we obtain for all $t\in [0,T]$,
    \begin{align}
    \label{EST:REG:7}
        &\frac 12 \intO \left| \frac 1h \int_{t-h}^t \Grad \mu(\tau) \dtau \right|^2 \dx
        + \frac{a_0}{2} \int_0^t \norm{\delth\Grad\varphi(s)}_{L^2(\Omega)}^2 \ds
        \notag\\[1ex]
        &\quad \le 
        C + 2CL \int_0^t \frac 12 \intO \left| \frac 1h \int_{s-h}^s \Grad\mu(\tau) \dtau \right|^2 \dx \ds.
    \end{align}
    Eventually, applying Gronwall's lemma, we conclude that for all $t\in [0,T]$,
    \begin{align}
    \label{EST:REG:8}
        \frac 12 \intO \left| \frac 1h \int_{t-h}^t \Grad \mu(\tau) \dtau \right|^2 \dx
        + \frac{a_0}{2} \int_0^t \norm{\delth\Grad\varphi(s)}_{L^2(\Omega)}^2 \ds
        \le C.
    \end{align}
    Using this information to bound the right-hand side of \eqref{EST:SUM2}, we further have
    \begin{align}
    \label{EST:REG:9}
        \int_0^t \norm{\delth\varphi(s)}_{L^2(\Omega)}^2 \ds 
        \le C
    \end{align}
    for all $t\in [0,T]$. 
    Combining \eqref{EST:REG:8} and \eqref{EST:REG:9}, we infer that $\delth\varphi$ is bounded in $L^2(0,T;H^1(\Omega))$ uniformly in $h$ and thus $\delt\varphi$ exists in the weak sense and belongs to $L^2(0,T;H^1(\Omega))$. As we already know from Theorem~\ref{THM:REG} that $\varphi\in L^2(0,T;H^2(\Omega))$, this proves
    \begin{align}
        \varphi \in H^1\big(0,T;H^1(\Omega)\big) \cap L^2(0,T;H^2(\Omega)).
    \end{align}
    Let now $(h_k)_{k\in\N} \subset (0,1)$ be an arbitrary sequence with $h_k\to 0$ as $k\to \infty$. Due to the extension of $\mu$ for negative times, we have $\mu\in L^2(-T,T;L^2(\Omega))$. 
    For any $k\in\N$, let
    \begin{equation*}
        \chi_k:\R\to\R,\quad \chi_k(s) := \frac{1}{h_k} \mathds{1}_{[-h_k,0]}(s) \revised{.}
    \end{equation*}
    This defines a Dirac sequence in the sense of \cite[Section~4.14]{Alt} and we have
    \begin{align*}
        I_{h_k}[\Grad\mu](t,\x) := \frac{1}{h_k} \int_{t-h_k}^t \Grad \mu(\tau,\x) \dtau = \big[\chi_k \ast \nabla\mu(\cdot,\x)\big](t)
    \end{align*}
    for almost all $t\in [0,T]$ and $\x\in\Omega$, where ``$\,\ast\,$'' denotes the convolution with respect to time. According to \cite[Theorem~4.15(2)]{Alt}, it thus holds
    \begin{align}
    \label{CONV:GMU:L2}
        I_{h_k}[\Grad\mu]
        = \chi_k \ast \nabla\mu \to \nabla\mu
        \quad\text{in \revised{$L^2(0,T;L^2(\Omega;\R^d))$}}
    \end{align}
    as $k\to\infty$.
    We further deduce from \eqref{EST:REG:8} that $I_{h_k}[\Grad\mu]$ is bounded in $L^\infty(0,T;L^2(\Omega;\R^d))$ uniformly in $k$. This implies that
    \begin{align*}
        I_{h_k}[\Grad\mu]
        \;\to\;
        \nabla\mu
        \quad\text{weakly-$^*$ in $L^\infty(0,T;L^2(\Omega;\R^d))$}
    \end{align*}
    as $k\to\infty$ along a non-relabeled subsequence.
    Here, \eqref{CONV:GMU:L2} was used to identify the \mbox{weak-$^*$} limit.
    In particular, we thus have
    \begin{align}
    \label{REG:GMU}
        \Grad \mu \in L^\infty\big(0,T;L^2(\Omega;\R^d)\big).
    \end{align}

    \textit{Step~2.} By means of elliptic regularity theory, we now show $\mu\in L^2(0,T;H^2(\Omega))$.

    In view of the regularities established in Step~1, there exists a Lebesgue null set $\mathcal N\subset [0,T]$ such that for all $t\in[0,T]\setminus \mathcal N$, it holds $\delt\varphi(t) \in H^1(\Omega)$, $\mu(t)\in L^2(\Omega)$ and
    \begin{align*}
        \intO \Grad\mu(t) \cdot \nabla \zeta \dx = - \intO \delt\varphi(t) \zeta \dx
    \end{align*}
    for all $\zeta\in H^1(\Omega)$. This means that $\mu(t)$ is a weak solution of the Poisson--Neumann problem
    \begin{align*}
        - \Delta \mu(t) = - \delt\varphi(t) \quad\text{in $\Omega$}, 
        \quad
        \Grad\mu(t) \cdot \n = 0 \quad\text{on $\Gamma$}.
    \end{align*}
    By means of elliptic regularity theory (see, e.g., \cite[Section~4]{McLean}) and \eqref{EST:REG:9}, we thus have $\mu(t) \in H^2(\Omega)$ for all $t\in[0,T]\setminus \mathcal N$ with
    \begin{align*}
        \int_0^T \norm{\mu(t)}_{H^2(\Omega)}^2 \dt \le C \int_0^T \norm{\delt\varphi(t)}_{L^2(\Omega)}^2 \dt \le C.
    \end{align*}
    This proves $\mu\in L^2(0,T;H^2(\Omega))$.
    \pagebreak[2]
    
    \textit{Step~3.} We next show the regularities $F'(\varphi)\in L^\infty(0,T;L^2(\Omega))$, $\mu\in L^\infty(0,T;H^1(\Omega))$ \revised{and $\varphi \in L^\infty(0,T;H^2(\Omega)) \cap C(\ov{\Omega_T})$}.

    If $F$ is a regular potential satisfying \eqref{ass:F1} with $p\le 4$ if $d=3$, we use the continuous embedding $H^1(\Omega) \emb L^{2(p-1)}(\Omega)$ to obtain
    \begin{align*}
        \intO \abs{F'\big(\varphi(t)\big)}^2 \dx 
        &\le C_{F'} \intO 1 + \abs{\varphi(t)}^{2(p-1)} \dx 
        \notag\\
        &
        \le C\big( 1 + \norm{\varphi(t)}_{L^\infty(0,T;H^1(\Omega))}^{2(p-1)} \big)
        \le C
    \end{align*}
    for almost all $t\in[0,T]$. This already implies $F'(\varphi)\in L^\infty(0,T;L^2(\Omega))$. 
    
    For the logarithmic potential, the regularity $F'(\varphi)\in L^\infty(0,T;L^2(\Omega))$ follows from the estimate \eqref{EST:F':LOG} in combination with \eqref{REG:GMU}.

    Using the weak formulation, the growth condition on $A'$ from \eqref{ass:A} as well as Young's inequality, we obtain for almost all $t\in [0,T]$, 
    \begin{align*}
        \norm{\mu(t)}_{L^2(\Omega)}^2 
        &= \intO A'\big(\Grad \varphi(t)\big) \cdot \Grad \mu(t) + F'\big(\varphi(t)\big) \mu(t) \dx
        \notag\\
        &\le C \intO \abs{\Grad\varphi(t)} \abs{\Grad \mu(t)} + \abs{F'\big(\varphi(t)\big)} \abs{\mu(t)} \dx
        \notag\\[1ex]
        &\le C \norm{\Grad\varphi}_{L^\infty(0,T;L^2(\Omega))}^2 + C \norm{\Grad\mu}_{L^\infty(0,T;L^2(\Omega))}^2
        \notag\\
        &\qquad + C \norm{F'(\varphi)}_{L^\infty(0,T;L^2(\Omega))}^2 + \frac 12 \norm{\mu(t)}_{L^2(\Omega)}^2.
    \end{align*}
    As we already know $\Grad\varphi,\Grad\mu \in L^\infty(0,T;L^2(\Omega;\R^d))$ and $F'(\varphi) \in L^\infty(0,T;L^2(\Omega))$, this proves $\mu \in L^\infty(0,T;L^2(\Omega))$ and thus, $\mu \in L^\infty(0,T;H^1(\Omega))$ directly follows.

    \revised{Recalling \eqref{EST:H2:PTW}, the regularities established above eventually imply
    \begin{align*}
        \norm{\varphi(t)}_{H^2(\Omega)}
        \le C \big( \norm{\mu(t)}_{L^2(\Omega)} 
        + \norm{F'\big(\varphi(t)\big)}_{L^2(\Omega)} 
        + \norm{\varphi(t)}_{L^2(\Omega)} +1 \big)
        \le C
    \end{align*}
    for almost all $t\in[0,T]$, which entails $\varphi \in L^\infty(0,T;H^2(\Omega))$.
    As we also have $\varphi \in H^1\big(0,T;H^1(\Omega)\big)$, we use the compact embedding $H^2(\Omega)\emb H^{7/4}(\Omega)$, the continuous embeddings $H^{7/4}(\Omega) \emb C(\ov\Omega)$ and $H^{7/4}(\Omega) \emb H^1(\Omega)$, and the Aubin--Lions--Simon lemma to deduce}
    \begin{align*}
        \revised{\varphi \in C\big([0,T];H^{7/4}(\Omega)) \emb C\big([0,T];C(\ov\Omega)\big) \cong C(\ov{\Omega_T}).}
    \end{align*}

    \textit{Step~4.} The next step is to prove $F'(\varphi) \in L^2(0,T;L^\infty(\Omega))$.

    Let us first consider the case where $F$ is the logarithmic potential. 
    Therefore, we consider $\varphi$ and $\mu$ as arbitrary representatives of their equivalence class. 
    For any $k\in\N$, we introduce the truncation
    \begin{align*}
        \sigma_k:\R\to\R, \quad
        \sigma_k(s) :=
        \begin{cases}
            1-\frac1k &\text{if $s\ge 1-\frac1k $},\\
            s &\text{if $\abs{s} < 1-\frac1k $},\\
            -1+\frac1k &\text{if $s\le -1+\frac1k $},
        \end{cases} 
    \end{align*}
    and we set $\varphi_k := \sigma_k\circ\varphi$. For any \revised{$q\ge 2$}
    and $k\in\N$, we further define
    \begin{align*}
        \revised{\eta_k := \bigabs{F_1'(\varphi_k)}^{q-2} F_1'(\varphi_k) \quad\text{if $q>2$},
        \quad\text{and}\quad
        \eta_k := F_1'(\varphi_k) \quad\text{if $q=2$}.}
    \end{align*}
    Hence, we have $\eta_k(t) \in H^1(\Omega )\cap L^\infty(\Omega)$ for almost all $t\in[0,T]$ and the gradient is given by
    \begin{align*}
        \Grad \eta_k := (q-1) \bigabs{F_1'(\varphi_k)}^{q-2} F_1''(\varphi_k) \Grad \varphi_k .
    \end{align*}
    Let now $k\in\N$ with $k\ge 2$ and \revised{$q\ge 2$} be arbitrary.
    Testing the weak formulation \eqref{DEF:WS:WF2} or \eqref{DEF:WS:WF2*}, respectively, with $\eta_k$, we obtain
    \begin{align}
        \label{WF:ETAK}
        &(q-1) \intO A'(\Grad\varphi)\cdot \Grad \varphi_k\, \bigabs{F_1'(\varphi_k)}^{q-2} F_1''(\varphi_k) \dx 
        \notag\\
        &\qquad
            + \intO F_1'(\varphi) \, F_1'(\varphi_k) \, \bigabs{F_1'(\varphi_k)}^{q-2} \dx
        \notag\\
        &\quad= 
            \intO \mu \, F_1'(\varphi_k) \, \bigabs{F_1'(\varphi_k)}^{q-2} \dx
            - \intO F_2'(\varphi) \, F_1'(\varphi_k) \, \bigabs{F_1'(\varphi_k)}^{q-2} \dx
    \end{align}
    a.e.~in $[0,T]$. By the definition of $\varphi_k$ and the monotonicity of $A'$, we have 
    \begin{align*}
        A'(\Grad\varphi)\cdot \Grad \varphi_k 
        = A'(\Grad\varphi_k)\cdot \Grad \varphi_k 
        \ge a_0 \abs{\Grad\varphi_k}^2
        \quad\text{a.e.~in $\Omega_T$}.
    \end{align*}
    As we further know $F_1''\ge 0$ (since $F_1$ is convex), we infer that the first summand on the left-hand side of \eqref{WF:ETAK} is non-negative. Since $k\ge 2$ and $F_1'$ is monotonically increasing, we further have
    \begin{align*}
        F_1'(\varphi_k)^2 \le F_1'(\varphi)\, F_1'(\varphi_k)
        \quad\text{a.e.~in $\Omega_T$},
    \end{align*}
    which can be used to estimate the second summand on the left-hand side of \eqref{WF:ETAK} from below. In view of \eqref{WF:ETAK} and the definition of $F_2$, we thus deduce 
    \begin{align}
        \label{EST:ETAK}
        &\intO \bigabs{F_1'(\varphi_k)}^{q} \dx
        \le \intO \big(\abs{\mu} + c_* \abs{\varphi} \big) \, \bigabs{F_1'(\varphi_k)}^{q-1} \dx
    \end{align}
    for some constant $c_*>0$ depending only on $F$.
    By means of Hölder's inequality, we infer \revised{
    \begin{align}
        \label{EST:ETAK:2*}
        \norm{F_1'(\varphi_k)}_{L^q(\Omega)}^{q}
        \le 
        \big( \norm{\mu(t)}_{L^q(\Omega)} + c_* \norm{\varphi(t)}_{L^q(\Omega)} \big) \,
        \norm{F_1'(\varphi_k)}_{L^q(\Omega)}^{q-1}
    \end{align} 
    a.e.~in $[0,T]$.}
    This directly implies
    \begin{align}
        \label{EST:ETAK:2}
        \norm{F_1'(\varphi_k)}_{L^q(\Omega)}^{q}
        \le 
        (1 + \abs{\Omega}) 
        \big( \norm{\mu(t)}_{L^\infty(\Omega)} + c_* \norm{\varphi(t)}_{L^\infty(\Omega)} \big) \,
        \norm{F_1'(\varphi_k)}_{L^q(\Omega)}^{q-1}
    \end{align}   
    a.e.~in $[0,T]$.
    Since $\varphi,\mu \in L^2(0,T;H^2(\Omega))$ and $H^2(\Omega) \emb L^\infty(\Omega)$, there exists a null set $\mathcal N\subset [0,T]$ such that for all $t\in [0,T]\setminus\mathcal N$,
    \begin{align}
    \label{EST:ct}
        c(t)&:=
        (1 + \abs{\Omega}) 
        \big( \norm{\mu(t)}_{L^\infty(\Omega)} + c_* \norm{\varphi(t)}_{L^\infty(\Omega)} \big)
        \notag\\
        &\phantom{:}\le 
        C \big( \norm{\mu(t)}_{H^2(\Omega)} + c_*\norm{\varphi(t)}_{H^2(\Omega)} \big) < \infty.
    \end{align}
    Using this estimate to bound the right-hand side of \eqref{EST:ETAK:2}, we conclude
    \begin{align}
    \label{EST:F1PK:P}
        \bignorm{F_1'\big(\varphi_k(t)\big)}_{L^q(\Omega)} \le c(t) < \infty
    \end{align}
    for all $q\ge 4$ and all $t\in [0,T]\setminus\mathcal N$, after possibly replacing $\mathcal N$ by a larger null set. As $c(t)$ is independent of $q$, this directly entails
    \begin{align}
    \label{EST:F1PK:INF}
        \bignorm{F_1'\big(\varphi_k(t)\big)}_{L^\infty(\Omega)} \le c(t) < \infty
    \end{align}
    for all $t\in [0,T]\setminus\mathcal N$. 


    
    As for all $t\in [0,T]\setminus\mathcal N$ and all $k\in\N$, the function $F_1'(\varphi_k(t))$ is measurable and non-negative, we apply Fatou's lemma to derive the estimate
    \begin{align}
    \label{EST:F1PK:FATOU}
        \intO \underset{k\to\infty}{\lim\inf} \; \bigabs{F_{1}'\big(\varphi_k(t)\big)}^2 \dx
        \le \underset{n\to\infty}{\lim\inf} \; \intO \bigabs{F_{1}'\big(\varphi_k(t)\big)}^2 \dx
        \le \abs{\Omega} c(t).
    \end{align}
    Moreover, it is straightforward to check that $\varphi_k(t) \to \varphi(t)$ a.e.~in $\Omega$ for all $t\in [0,T]\setminus\mathcal N$, after possibly replacing $\mathcal N$ by a larger null set. 

    Let now $t\in [0,T]\setminus\mathcal N$ be arbitrary.
    Since $\abs{\varphi_k(t)} < 1$ a.e.~in $\Omega$ for all $k\in\N$, this already implies that $\varphi(t) \in [-1,1]$ a.e.~in $\Omega$.
    However, if $\abs{\varphi(t,\x)} = 1$ for some $x\in \Omega$, we have 
    $\abs{\varphi_k(t,\x)} \to 1$ and thus $\abs{F_1'(\varphi_k(t,\x))} \to +\infty$ as $k\to\infty$. We thus conclude that the set 
    \begin{align*}
        \mathcal M_t := \big\{ x\in\Omega \,\big\vert\, \abs{\varphi(t,\x)} = 1 \big\} 
    \end{align*}
    has Lebesgue measure zero as otherwise, this would contradict \eqref{EST:F1PK:FATOU}. Consequently, we have $\varphi(t) \in (-1,1)$ a.e.~in $\Omega$ and
    \begin{align}
    \label{CONV:F1P:PTW}
        F_1'\big(\varphi_k(t)\big) \to F_1'\big(\varphi(t)\big) \quad\text{a.e.~in $\Omega$ as $k\to\infty$}.
    \end{align}

    
    In combination with \eqref{EST:F1PK:INF}, this proves
    \begin{align}
    \label{EST:F1P:INF}
        \bignorm{F_1'\big(\varphi(t)\big)}_{L^\infty(\Omega)} \le c(t) < \infty
    \end{align}
    for all $t\in [0,T]\setminus\mathcal N$. We now recall \eqref{EST:ct} as well as the regularities $\varphi,\mu \in L^2(0,T;H^2(\Omega))$. Taking the square on both sides of \eqref{EST:F1P:INF} and integrating with respect to time from $0$ to $T$, we eventually conclude $F_1'(\varphi) \in L^2(0,T;L^\infty(\Omega))$. As a direct consequence, we have $F'(\varphi) \in L^2(0,T;L^\infty(\Omega))$.
    
    Let us now consider the case where $F$ is a regular potential. After fixing representatives $\varphi$ and $\mu$ of their equivalence class, we introduce the truncation 
    \begin{align*}
        \sigma_k:\R\to\R, \quad
        \sigma_k(s) :=
        \begin{cases}
            k &\text{if $s\ge k$},\\
            s &\text{if $\abs{s} < k$},\\
            -k &\text{if $s\le -k$},
        \end{cases} 
    \end{align*}
    and we set $\varphi_k:=\sigma_k\circ \varphi$ for all $k\in\N$. Proceeding exactly as for the logarithmic potential, we show that $F_1'\big(\varphi_k(t)\big)$ satisfies \eqref{EST:F1PK:INF} for all $t\in [0,T]\setminus\mathcal N$, where $\mathcal N\subset [0,T]$ is some null set and $c(t)$ is the constant defined in \eqref{EST:ct}. As the regular potential is assumed to fulfill \eqref{ass:F3}, we know that $F_1':\R\to\R$ is continuous. Hence, since $\varphi_k(t) \to \varphi(t)$ a.e.~in $\Omega$ for all $t\in [0,T]\setminus\mathcal N$, we deduce the convergence \eqref{CONV:F1P:PTW}. This entails \eqref{EST:F1P:INF} and by proceeding as above, we conclude the claim.

    \textit{Step~5:} From now on, let $F$ be the logarithmic potential.
    We now show that the strict separation property \eqref{EST:SEPPROP} holds for almost all $t\in[0,T]$.

    As a consequence of \eqref{EST:F1P:INF}, there exist null sets $\mathcal N\subset [0,T]$ and $\mathcal O \subset \Omega$ such that for all $(t,\x) \in \Omega_T\setminus (\mathcal N\times\mathcal O)$, 
    it holds $\varphi(t,\x) \in (-1,1)$ and
    \begin{align}
    \label{EST:F1P:PTW}
        \bigabs{\ln\big( 1 + \varphi(t,\x) \big) - \ln\big( 1 - \varphi(t,\x) \big)}
        = \Big|\frac{2}{\theta} \, F_1'\big(\varphi(t,\x)\big)\Big| \le \frac{2 c(t)}{\theta} .
    \end{align}
    Let now $(t,\x) \in \Omega_T\setminus (\mathcal N\times\mathcal O)$ be arbitrary.
    
    If $\varphi(t,\x) \in [0,1)$, we have $\ln\big( 1 + \varphi(t,\x) \big) \ge 0$ and $-\ln\big( 1 - \varphi(t,\x) \big) \ge 0$. Hence,
    \begin{align*}
        - \ln\big( 1 - \varphi(t,\x) \big) 
        \le \frac{2 c(t)}{\theta} .
    \end{align*}
    By a straightforward computation, we thus infer
    \begin{align}
    \label{EST:F1P:PTW+}
        0 
        \le \varphi(t,\x)  
        \le 1- \delta(t),
        \quad\text{with}\quad
        \delta(t) := \exp\left(-\frac{2 c(t)}{\theta}\right) \in (0,1].
    \end{align}
    If $\varphi(t,\x) \in (-1,0)$, a similar computation yields $- \big(1 - \delta(t) \big) \le \varphi(t,\x) < 0$.
    
    In summary, we have thus shown that
    \begin{align*}
        \norm{\varphi(t)}_{L^\infty(\Omega)} \le 1 - \delta(t)
    \end{align*}
    for all $t\in [0,T]\setminus \mathcal N$. This verifies the strict separation property \eqref{EST:SEPPROP} for almost all $t\in[0,T]$.

    \revised{
    \textit{Step~6:} As the final step, we verify the uniform separation property \eqref{EST:SEPPROP:UNI} in the case $d=2$.

    Recalling the definition of $F_1$ in the context of the logarithmic potential (see \eqref{DEC:LOG}), we observe
    \begin{align}
        \label{EST:F''}
        \abs{F_1''(s)} \le \theta \mathrm{e}^{\frac2\theta  \abs{F'(s)}}
        \quad\text{for all $s\in\R$}.
    \end{align}   
    Since $d=2$, we know that for all $u\in H^1(\Omega)$ and $r\in [2,\infty)$, it holds
    \begin{align*}
        \norm{u}_{L^r(\Omega)} 
        &\le C_\Omega \sqrt{r}\, \norm{u}_{H^1(\Omega)} 
    \end{align*}
    for some constant $C_\Omega>0$ depending only on $\Omega$.
    This inequality can, for instance, be found in \cite[p.~479]{Trudinger}. Consequently, we infer from \eqref{EST:ETAK:2*} that
    \begin{align*}
        \bignorm{F_1'\big(\varphi_k(t)\big)}_{L^q(\Omega)} \le C\sqrt{q} \big(\norm{\mu}_{L^\infty(0,T;H^1(\Omega))} + c_* \norm{\mu}_{L^\infty(0,T;H^1(\Omega))}\big)
        \le C \sqrt{q}
    \end{align*}
    for almost all $t\in [0,T]$, all $q\ge 2$, and all $k\in\N$ with $k\ge 2$.
    Arguing similarly as in Step~5, we conclude
    \begin{align}
        \label{EST:F':Q}
        \bignorm{F_1'\big(\varphi(t)\big)}_{L^q(\Omega)} \le C\sqrt{q}
    \end{align}    
    for almost all $t\in [0,T]$ and all $q\ge 2$.
    Based on this inequality, we proceed as in the proof of \cite[Theorem~3.1]{Gal2023} to obtain the estimate
    \begin{align}
        \label{EST:F':EXP}
        \intO \abs{\mathrm{e}^{\frac2\theta \abs{F'(\varphi(t))}}}^m \dx \le C(m)
    \end{align}
    for all $m\in\N$, almost all $t\in [0,T]$ and some constant $C(m)$ depending only on $m$ and the constant $C$ from \eqref{EST:F':Q} (cf.~\cite[Eq.~(26)]{Gal2023} with the constants therein being chosen as $C_1 = \frac{2}{\theta}$, $C_2=0$ and $\beta=1$).
    In view of \eqref{EST:F''}, inequality \eqref{EST:F':EXP} written for $m=6$ directly implies
    \begin{align}
        \bignorm{F_1''(\varphi)}_{L^\infty(0,T;L^6(\Omega))} \le C.
    \end{align}
    Employing the chain rule and the continuous embedding $H^2(\Omega)\emb W^{1,6}(\Omega)$, we thus obtain
    \begin{align}
        \bignorm{\Grad F_1'(\varphi)}_{L^\infty(0,T;L^3(\Omega))} 
        &\le \bignorm{F_1''(\varphi)}_{L^\infty(0,T;L^6(\Omega))} \norm{\Grad\varphi}_{L^\infty(0,T;L^6(\Omega))}
        \notag\\
        &\le C \norm{\varphi}_{L^\infty(0,T;H^2(\Omega))}
        \le C.
    \end{align}
    Using \eqref{EST:F':Q} written for $q=3$ as well as the continuous embedding $W^{1,3}(\Omega) \emb C(\ov\Omega)$, we infer
    \begin{align}
    \label{EST:F':UNI}
        \bignorm{F_1'\big(\varphi(t)\big)}_{L^\infty(\Omega)} 
        \le \norm{F_1'(\varphi)}_{L^\infty(0,T;W^{1,3}(\Omega))}
        \le C =: C^*
    \end{align}
    for almost all $t\in [0,T]$. As we have already shown $\varphi\in C(\ov{\Omega_T})$, we know that \eqref{EST:F':UNI} actually holds true for \textit{all} $t\in[0,T]$. By arguing as in Step~5, we eventually conclude
    \begin{align*}
        \abs{\varphi(t,\x)} \le 1 - \delta^*
        \quad\text{with}\quad
        \delta^* := \exp\left(-\frac{2C^*}{\theta}\right) \in (0,1],
    \end{align*}
    for all $(t,\x) \in [0,T]$, where $C^*$ is the constant from \eqref{EST:F':UNI}. Therefore, the uniform separation property \eqref{EST:SEPPROP:UNI} is verified.
    }

    This means that all claims are established and thus, the proof is complete.   
\end{proof}



\section*{Acknowledgement}
The authors gratefully acknowledge the support by the RTG 2339 “Interfaces, Complex Structures, and Singular Limits” of the Deutsche Forschungsgemeinschaft (DFG, German Research Foundation). The authors also acknowledge continuous inspiration by the work of Pierluigi Colli who, according to MathSciNet, published 47 papers related to the Cahn--Hilliard equation. We also thank the anonymous referee for valuable comments which helped to improve the paper significantly.




\footnotesize

\bibliographystyle{abbrv}
\bibliography{GKW1}

\begin{thebibliography}{10}

\bibitem{Alfaro2010motion}
M.~Alfaro, H.~Garcke, D.~Hilhorst, H.~Matano, and R.~Sch{\"a}tzle.
\newblock Motion by anisotropic mean curvature as sharp interface limit of an
  inhomogeneous and anisotropic {Allen--Cahn} equation.
\newblock {\em Proc. R. Soc. Edinb. A}, 140(4):673--706, 2010.

\bibitem{Alt}
H.~W. Alt.
\newblock {\em {Linear Functional Analysis - An Application-Oriented
  Introduction}}.
\newblock Springer, London, 2016.

\bibitem{Ambrosio-DalMaso}
L.~Ambrosio and G.~Dal~Maso.
\newblock A general chain rule for distributional derivatives.
\newblock {\em Proc. Amer. Math. Soc.}, 108(3):691--702, 1990.

\bibitem{BGN2013stable}
J.~W. Barrett, H.~Garcke, and R.~N{\"u}rnberg.
\newblock On the stable discretization of strongly anisotropic phase field
  models with applications to crystal growth.
\newblock {\em ZAMM Z. Angew. Math. Mech.}, 93(10-11):719--732, 2013.

\bibitem{vch}
J.~W. Barrett, H.~Garcke, and R.~N\"urnberg.
\newblock Stable phase field approximations of anisotropic solidification.
\newblock {\em IMA J. Numer. Anal.}, 34(4):1289--1327, 2014.

\bibitem{Barroso1994}
A.~C. Barroso and I.~Fonseca.
\newblock Anisotropic singular perturbations---the vectorial case.
\newblock {\em Proc. Roy. Soc. Edinburgh Sect. A}, 124(3):527--571, 1994.

\bibitem{Bellettini2005}
G.~Bellettini, A.~Braides, and G.~Riey.
\newblock Variational approximation of anisotropic functionals on partitions.
\newblock {\em Ann. Mat. Pura Appl. (4)}, 184(1):75--93, 2005.

\bibitem{Bouchitte1990}
G.~Bouchitt\'{e}.
\newblock Singular perturbations of variational problems arising from a
  two-phase transition model.
\newblock {\em Appl. Math. Optim.}, 21(3):289--314, 1990.

\bibitem{MR3918378}
L.~Cherfils, A.~Miranville, and S.~Peng.
\newblock Higher-order anisotropic models in phase separation.
\newblock {\em Adv. Nonlinear Anal.}, 8(1):278--302, 2019.

\bibitem{Dziwnik17anisotropic}
M.~Dziwnik, A.~M{\"u}nch, and B.~Wagner.
\newblock An anisotropic phase-field model for solid-state dewetting and its
  sharp-interface limit.
\newblock {\em Nonlinearity}, 30(4):1465, 2017.

\bibitem{Elliott96limit}
C.~M. Elliott and R.~Sch{\"a}tzle.
\newblock The limit of the anisotropic double-obstacle {Allen--Cahn} equation.
\newblock {\em Proc. R. Soc. Edinb. A}, 126(6):1217--1234, 1996.

\bibitem{Evans2010}
L.~C. Evans.
\newblock {\em Partial differential equations}, volume~19 of {\em Graduate
  Studies in Mathematics}.
\newblock American Mathematical Society, second edition, 2010.

\bibitem{Fife93}
P.~C. Fife.
\newblock Models for phase separation and their mathematics.
\newblock {\em Electron. J. Differential Equations}, pages No. 48, 26, 2000.

\bibitem{Gal2023}
C.~G. Gal, A.~Giorgini, and M.~Grasselli.
\newblock The separation property for 2{D} {C}ahn-{H}illiard equations: local,
  nonlocal and fractional energy cases.
\newblock {\em Discrete Contin. Dyn. Syst.}, 43(6):2270--2304, 2023.

\bibitem{Garcke2003}
H.~Garcke.
\newblock On {C}ahn-{H}illiard systems with elasticity.
\newblock {\em Proc. Roy. Soc. Edinburgh Sect. A}, 133(2):307--331, 2003.

\bibitem{GarckeDMV}
H.~Garcke.
\newblock Curvature driven interface evolution.
\newblock {\em Jahresber. Dtsch. Math.-Ver.}, 115(2):63--100, 2013.

\bibitem{Garcke2020}
H.~Garcke and P.~Knopf.
\newblock Weak solutions of the {C}ahn-{H}illiard system with dynamic boundary
  conditions: a gradient flow approach.
\newblock {\em SIAM J. Math. Anal.}, 52(1):340--369, 2020.

\bibitem{GKNZ}
H.~Garcke, P.~Knopf, R.~N\"{u}rnberg, and Q.~Zhao.
\newblock A {D}iffuse-{I}nterface {A}pproach for {S}olid-{S}tate {D}ewetting
  with {A}nisotropic {S}urface {E}nergies.
\newblock {\em J. Nonlinear Sci.}, 33(2):Paper No. 34, 2023.

\bibitem{Garcke2023}
H.~Garcke, K.~F. Lam, R.~N\"{u}rnberg, and A.~Signori.
\newblock Overhang penalization in additive manufacturing via phase field
  structural topology optimization with anisotropic energies.
\newblock {\em Appl. Math. Optim.}, 87(3):Paper No. 44, 50, 2023.

\bibitem{Gilbarg2001}
D.~Gilbarg and N.~S. Trudinger.
\newblock {\em Elliptic Partial Differential Equations of Second Order}.
\newblock Springer Berlin, Heidelberg, 2001.

\bibitem{Graser2013time}
C.~Gr{\"a}ser, R.~Kornhuber, and U.~Sack.
\newblock Time discretizations of anisotropic {Allen--Cahn} equations.
\newblock {\em IMA J. Numer. Anal.}, 33(4):1226--1244, 2013.

\bibitem{laux2022diffuseinterface}
T.~Laux, K.~Stinson, and C.~Ullrich.
\newblock Diffuse-interface approximation and weak-strong uniqueness of
  anisotropic mean curvature flow.
\newblock Preprint: \href{https://arxiv.org/abs/2212.11939}{arXiv:2212.11939}
  [math.AP], 2022.

\bibitem{McLean}
W.~McLean.
\newblock {\em Strongly elliptic systems and boundary integral equations}.
\newblock Cambridge University Press, Cambridge, 2000.

\bibitem{Miranvillebook}
A.~Miranville.
\newblock {\em The {C}ahn-{H}illiard equation}, volume~95 of {\em CBMS-NSF
  Regional Conference Series in Applied Mathematics}.
\newblock Society for Industrial and Applied Mathematics (SIAM), Philadelphia,
  PA, 2019.
\newblock Recent advances and applications.

\bibitem{Owen1991}
N.~C. Owen and P.~Sternberg.
\newblock Nonconvex variational problems with anisotropic perturbations.
\newblock {\em Nonlinear Anal.}, 16(7-8):705--719, 1991.

\bibitem{Ratz2006surface}
A.~R{\"a}tz, A.~Ribalta, and A.~Voigt.
\newblock Surface evolution of elastically stressed films under deposition by a
  diffuse interface model.
\newblock {\em J. Comput. Phys.}, 214(1):187--208, 2006.

\bibitem{Taylor94linking}
J.~E. Taylor and J.~W. Cahn.
\newblock Linking anisotropic sharp and diffuse surface motion laws via
  gradient flows.
\newblock {\em J. Stat. Phys.}, 77(1):183--197, 1994.

\bibitem{Torabi209}
S.~Torabi, J.~Lowengrub, A.~Voigt, and S.~Wise.
\newblock A new phase-field model for strongly anisotropic systems.
\newblock {\em Proc. R. Soc. Lond. Secr. A Math. Phys. Eng. Sci.},
  465(2105):1337--1359, 2009.

\bibitem{Trudinger}
N.~S. Trudinger.
\newblock On imbeddings into {O}rlicz spaces and some applications.
\newblock {\em J. Math. Mech.}, pages 473--483, 1967.

\end{thebibliography}

\end{document}